\documentclass[11pt]{article}

\usepackage{epsfig,amsmath,amssymb,amsthm,graphics,verbatim,amsfonts,subfigure,psfrag}
\usepackage[letterpaper,margin=1in]{geometry}

\let\oldbibliography\thebibliography
\renewcommand{\thebibliography}[1]{%
  \oldbibliography{#1}%
  \setlength{\itemsep}{0.5mm}%
}

\newtheorem{thm}{Theorem}[section]
\newtheorem{cor}[thm]{Corollary}
\newtheorem{prop}[thm]{Proposition}

\def\R{\mathbb{R}}

\def\I{\infty}

\def\txtd{{\textnormal{d}}}
\def\txte{{\textnormal{e}}}

\def\txts{{\textnormal{s}}}

\def\txtD{{\textnormal{D}}}
\def\txtT{{\textnormal{T}}}
\def\txtN{{\textnormal{N}}}

\newcommand{\be}{\begin{equation}}
\newcommand{\ee}{\end{equation}}
\newcommand{\bea}{\begin{eqnarray}}
\newcommand{\eea}{\end{eqnarray}}
\newcommand{\beann}{\begin{eqnarray*}}
\newcommand{\eeann}{\end{eqnarray*}}
\newcommand{\benn}{\begin{equation*}}
\newcommand{\eenn}{\end{equation*}}

\def\ra{\rightarrow}
\def\I{\infty}

\def\Rey{\mbox{Re}}   

\newcommand{\cC}{{\mathcal C}}  
\newcommand{\cE}{{\mathcal E}}  
\newcommand{\cF}{{\mathcal F}}  
\newcommand{\cG}{{\mathcal G}}  
\newcommand{\cI}{{\mathcal I}}  
\newcommand{\cK}{{\mathcal K}}  
\newcommand{\cM}{{\mathcal M}}  
\newcommand{\cO}{{\mathcal O}}  
\newcommand{\cS}{{\mathcal S}}  
\newcommand{\cT}{{\mathcal T}}  

\usepackage{amsmath} \allowdisplaybreaks[1]
\usepackage{amsfonts}

\begin{document}

\title{Tracking Particles in Flows near Invariant Manifolds via Balance Functions}

\author{Christian Kuehn\footnotemark[2]\ \footnotemark[5]
\and Francesco Roman\`o\footnotemark[3]\
\and Hendrik C. Kuhlmann\footnotemark[3]}

\maketitle

\renewcommand{\thefootnote}{\fnsymbol{footnote}}

\footnotetext[2]{Technical Unversity Munich, Fakult\"at f\"ur Mathematik, 
Boltzmannstr.~3, 85748 Garching bei M\"unchen, Germany}
\footnotetext[3]{Institute of Fluid Mechanics and Heat Transfer, 
Vienna University of Technology, Getreidemarkt 9, 1060 Vienna, Vienna, Austria}
\footnotetext[5]{CK acknowledges support via an APART fellowship of the Austrian 
Academy of Sciences ({\"{O}AW}) and via a Lichtenberg Professorship of the 
VolkswagenFoundation. Furthermore,
CK would like to thank Peter Szmolyan and Daniel Karrasch for general discussions
regarding non-autonomous dynamical systems, and Armin Rainer for a clarifying
remark regarding differentiability of certain parametrized eigenvalues.}

\begin{abstract}
Particles moving inside a fluid near, and interacting with, invariant manifolds is
a common phenomenon in a wide variety of applications. One elementary question is
whether we can determine once a particle has entered a neighbourhood of an invariant
manifold, when it leaves again. Here we approach this problem mathematically by
introducing balance functions, which relate the entry and exit points of a particle
by an integral variational formula. We define, study, and compare different
natural choices for balance functions and conclude that an efficient compromise
is to employ normal infinitesimal Lyapunov exponents. We apply our results to two
different model flows: a regularized solid-body rotational flow and the asymmetric
Kuhlmann--Muldoon model developed in the context of liquid bridges. Furthermore, we
employ full numerical simulations of the Navier-Stokes equations of a two-way coupled 
particle in a shear--stress-driven cavity to test balance functions for a particle
moving near an invariant wall. In conclusion, our theoretically-developed framework
seems to be applicable to models as well as data to understand particle motion near
invariant manifolds.
\end{abstract}


\section{Introduction}

Tracking particles inside a flow is a topic of general importance in a wide variety of
applications, ranging from small scales in microfluidics~\cite{StoneStroockAjdari}, to
mesoscale problems in sedimentation~\cite{VasseurCox}, to large oceanic scales~\cite{McCave},
and even astrophysical scales~\cite{ClarkeCarswell}. Many results for particle motion
are immediately useful for very classical problems such as turbulent pipe
flow~\cite{YoungLeeming} as well as newly arising recent challenges, for example in
the search for parts of crashed airplanes~\cite{RosebrockOkeCarroll}. The dynamics of
particles within certain domain and flow geometries has been studied
intensively, e.g., for particles moving near walls~\cite{CoxHsu}, colliding with
walls~\cite{JosephZenitHuntRosenwinkel}, with respect to rotation-induced
forces~\cite{RubinowKeller} or changing viscosity~\cite{Haber}, or via quite general 
partial differential and integral equations of motion for particles~\cite{MaxeyRiley,PowerPower}. 
Another important branch of current research is to study separation structures in
flows, which act as transport barriers, i.e., particles cannot pass through them.
The definition~\cite{HallerYuan,Haller5,ShaddenLekienMarsden},
analysis~\cite{Haller4,FroylandPadbergEnglandTreguier}, and
computation~\cite{GreenRowleyHaller,LekienRoss} of these Lagrangian coherent structures has
received a lot of attention; see also the references in the recent review~\cite{Haller3}
and the focus issue~\cite{PeacockDabiri}. A directly related topic are invariant
manifolds for general dynamical systems~\cite{Fenichel1,HirschPughShub}. The
autonomous (time-independent vector field) has to be studied first but since 
transport barriers in flows are often viewed in terms of non-autonomous dynamical
systems~\cite{KloedenRasmussen,Rasmussen}, i.e., as time-dependent invariant
manifolds~\cite{AulbachRasmussenSiegmund,BranickiWiggins,DucSiegmund}, we provide
a setup, which also works for non-autonomous cases.

\begin{figure}[htbp]
\psfrag{t0}{$t=t_0$}
\psfrag{t1}{$t\in \cI$}
\psfrag{t2}{$t=T$}
\psfrag{M}{$\cM$}
\psfrag{E}{$\cE$}
	\centering
		\includegraphics[width=0.9\textwidth]{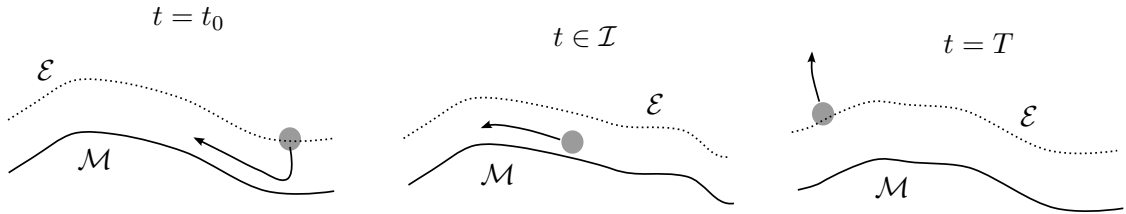}	
		\caption{\label{fig:01}Sketch of the basic situation studied in this paper. A
		particle (grey disc) moves near an invariant manifold $\cM$ inside a flow. The question
		is how to relate the entry and exit points into a small tubular neighbourhood $\cE$ of
		$\cM$; note that $\cM=\cM(t)$ and $\cE=\cE(t)$ could be time-dependent for general fluids.
		The key calculation is to relate the time $t_0$ at entry, via a drift time near $\cM$ inside
		some compact time domain $\cI$, to the final exit time $T$.}
\end{figure}

Despite considerable progress, there seems to be no completely general mathematical
framework available for some relatively elementary-looking problems involving particles.
One such question is, how to determine the entry-exit relationship for a particle near
an invariant structure; see Figure~\ref{fig:01}. In Section~\ref{sec:motivation}, we provide
further background from experiments, which motivated our work on entry-exit problems.

In this paper, we develop a mathematical framework based on balancing ideas, i.e., we 
ask the question which quantity is adequate
to compute (or measure) the entry-exit relationship of a particle moving near a
time-dependent invariant manifold. The main idea of this work is motivated by a
classical problem in multiple time scale systems~\cite{KuehnBook}. In
this context, the invariant manifolds of interest are slow manifolds~\cite{Jones,Fenichel4}.
Fast transverse dynamics to slow manifolds may be attracting as well as repelling in
certain directions. Of particular interest are trajectories which first get attracted
to slow manifolds, pass near certain bifurcation points of the fast dynamics, and
eventually get repelled from a neighbourhood of the slow manifold. This problem has been
studied intensively during recent years in the context of various bifurcations such as
fold points~\cite{DumortierRoussarie,KruSzm2}, Hopf
bifurcation~\cite{Neishtadt1,BaerErneuxRinzel}, as well as more complicated
singularities~\cite{Wechselberger1,KuehnUM}. One key idea in this context is the
so-called entry-exit (or way-in/way-out, or input-output) map which can be used to
calculate the relationship between entry and exit from a neighbourhood of a slow manifold
involving a stability change. Here we substantially extend this idea to the case
of fluid flows, as well as to arbitrary dimensions, and to various different possible
maps involving entry-exit from time-dependent invariant manifolds. Then we apply it to
several examples motivated by recent experiments and models of particle-surface
interaction~\cite{MuldoonKuhlmann,TanakaKawamuraUenoSchwabe,KuhlmannMuldoon} as well as
to full two-dimensional Navier-Stokes simulations for a particle motion fully coupled 
to the fluid. Our main results (in non-technical form) are the following:\medskip

\textit{Section~\ref{sec:basic}:} We introduce several natural notions for balance
  functions to determine entry-exit relationships based upon instantaneous
  eigenvalues, finite-time Lyapunov exponents (FTLE)~\cite{LekienRoss,Haller5}, fast-slow 
	scaling~\cite{KuehnBook} and normal infinitesimal Lyapunov exponents 
	(NILE)~\cite{HallerSapsis1}. We prove suitable differentiability
  as well as continuous dependence on data for the balance functions and develop
	perturbation theory results.\medskip

\textit{Section~\ref{sec:passive}:} To evaluate which concept is most promising, we
 investigate two test models (solid-body rotation and the Kuhlmann--Muldoon model), where
 the particle does not actively influence the flow. Three concepts turn out to have
 some drawbacks while NILE seems sufficiently easy to compute and conceptually the
 most adequate for our purpose.\medskip

\textit{Section~\ref{sec:creep}:} We consider direct numerical simulation (DNS) of the
Navier-Stokes equations for a two-way coupled particle in a shear--stress-driven cavity. In this
context, there is an invariant free-surface in the cavity, which is tracked by the particle for
a transient period. We use the DNS data to compute the NILE balance function. The results
show that, although the NILE balance function is  only an integrated locally linearized
estimator near the boundary, it performs very similar to a full  nonlinear analysis of the
particle velocities.\medskip

\textit{Section~\ref{sec:outlook}:} We conclude with an outlook of potential applications
and future challenges. We believe many interesting directions could be pursued beyond the
fundamental groundwork presented in this study. In particular, there are interesting
questions from theoretical, modelling, and data analysis perspectives.

\section{Motivation and Further Background}
\label{sec:motivation}

An example for an entry-exit problem of current 
interest~\cite{PushkinMelnikovShevtsova,KuhlmannMuldoon2,MelnikovPushkinShevtsova,KuhlmannMuldoon1} 
are particle accumulation structures 
(PAS)~\cite{TanakaKawamuraUenoSchwabe,SchwabeMizevUdhayasankarTanaka}. The phenomenon during 
which particles suspended in a liquid are rapidly segregating.\medskip 

\textbf{Remark:} The final version of this paper contains a figure 
from~\cite{SchwabeMizevUdhayasankarTanaka} illustrating this experiment, which we 
cannot reproduce here on the arXiv.\medskip 

A natural transport barrier for fully wetted particles is the liquid--gas 
interface of a cylindrical droplet (liquid bridge) in which a flow is driven by tangential 
shear stresses created by temperature-induced axial variations of the surface 
tension~\cite{Kuhlmann}. According to~\cite{HofmannKuhlmann} and~\cite{MuldoonKuhlmann} PAS 
are attractors for the motion of small particles in a three-dimensional steady flow which 
are created in a close neighbourhood of the invariant manifold represented by the liquid--gas 
interface. The mechanism involves the particle--free-surface interaction during which a particle 
enters a certain neighbourhood of the invariant manifold, experiences extra forces due to the 
presences of the interface, and exits to the bulk. If the properties of the entry--exit process 
are such that particles enter from a region of chaotic streamlines and exit to a region of 
regular streamlines of the steady three-dimensional flow a periodic attractor is typically 
created, in particular, if the particle inertia is small.\medskip 

A direct numerical simulation of this process based on the Navier--Stokes equations 
would require enormous computing resources, because all length scales must be properly 
resolved, ranging from the scale of the droplet over the scale of the particles down to 
the scale of the lubrication gap between a particle's surface and the liquid--gas interface. 
For that reason suitable entry-exit relations which go beyond inelastic collision 
models~\cite{HofmannKuhlmann} would be extremely helpful for a realistic and economic 
simulation of PAS using one-way coupling, e.g.\ in the framework of the 
Maxey--Riley equation~\cite{MaxeyRiley}; we refer to Section~\ref{ssec:MuldoonKuhlmann} 
for the analysis of a first model using balance functions in this direction. 

Analytical approaches would also be a first step to check whether there is a stretching 
mechanism near the surface so some particles stay there a lot longer than others although 
they entered in nearby regions. Such a stretching mechanism is key for the generation of 
chaotic dynamics in many situations including problems involving the classical Smale 
horseshoe~\cite{GH}. Furthermore, analytical entry-exit relationships could be used for 
perturbation theory (see Section~\ref{ssec:perturb}) as well as for extrapolation of 
particle locations based upon previous particle data (see Section~\ref{sec:creep}).

Further real-world examples for entry-exit problems are mixing processes, either in a 
rotating drum~\cite{OttinoKhakhar} or by use of impellers~\cite{GogateBeenackersPandit}. In 
these devices suspended particles enter and exit a neighbourhood of a moving wall. Since 
the fluid flow typically becomes ergodic from the boundaries, the entry-exit properties 
may play a key role for the mixing of suspended particles which move different from the 
flow, in particular near the boundaries. Other systems of interest are fluidized 
beds~\cite{XuYu} with wall effects becoming important in micro-fluidized beds~\cite{LiuXuGao}. 
Finally, dynamical entry-exit problems should be of importance for the deposition of 
evaporating micro-droplets in the human airways~\cite{ZhangKleinstreuerKimCheng}.

\section{Basic Mathematical Framework}
\label{sec:basic}

Consider the ordinary differential equation (ODE)
\be
\label{eq:ODE}
\frac{\txtd z}{\txtd t}=:z'=h(z(t),t),\qquad z\in\R^d,~t\in\R,
\ee
for $d\geq 2$ and with initial condition $z(t_0)=:z_0\in \cK$, where $\cK$ is a
compact subset of $\R^d$. We always assume that $h$
is smooth, i.e., $C^\I=C^\I(\R^{d+1},\R^d)$; this assumption could be weakened
but it will be convenient for the applications we have in mind. Denote the flow
associated with~\eqref{eq:ODE} by
\be
\phi_{t_0}^t:z_0\mapsto z(t)=z(t;t_0,z_0).
\ee
Let $\cM(t)\subset \R^d$ be an $m$-dimensional smooth invariant manifold with
\benn
\phi_{t_0}^t(\cM(t_0))\subset \cM(t),\qquad \forall t\in[t_1,t_2]=:\cI,
\eenn
for some $t_1<t_2$, with $t_0,t\in \cI$, and for $t\geq t_0$; we only
consider the dynamics on the compact time interval $\cI$ as we are interested
in finite-time non-autonomous dynamics. Furthermore, we emphasize that we 
do {\it not} require any particular type of invariant manifold such as normally 
hyperbolic~\cite{Fenichel4} or normally elliptic. However, one has to pick a 
relevant invariant manifold (e.g.~a domain boundary, a special co-dimension 
one surface, an interface between two physically/biologically relevant regions).
Here we do not discuss this choice as it does depend upon the application but
we refer to the outlook in Section \ref{sec:outlook}, where one goal would be
to try our methods for different definitions of Lagrangian coherent structures.

Let $\txtT_p\cM(t)$ and $\txtN_p\cM(t)$ denote the tangent 
and normal spaces at $p$ to
$\cM(t)$. Let $\gamma=\gamma(t;t_0,x_0)$ denote a trajectory in $\cM(t)$, i.e., we
require $\gamma(t)\subset \cM(t)$ for all $t\in \cI$. Suppose $\gamma$ is smooth,
i.e., $\gamma\in C^\I(\cI\times \cI\times \cK,\R^d)$. Then we define the function
\be
t\mapsto \cF(t)=\tilde{\cF}(\gamma(t;t_0,z_0)), \qquad \cF:\cI\ra \R,
\ee
where the choice for $\tilde{\cF}$ respectively the construction of $\cF$ will be 
discussed below. As yet, $\cF$ is quite a general function as it just maps a
time $t$ to $\R$. The idea is to \emph{associate the location of zeros of
$\cF$ with certain balancing properties of trajectories contained in $\cM(t)$,
to relate the entrance and exit points for particles moving near $\cM(t)$}. We 
emphasize that we only need a reference trajectory $\gamma(t)\subset \cM(t)$, not
the entire manifold $\cM(t)$ for our calculation.

\subsection{Instantaneous Eigenvalues}
\label{ssec:evalinst}

A naive first guess to study the dynamical properties near $\cM(t)$ is to
consider instantaneous eigenvalues. Consider the non-autonomous linear system
\be
Z'=[\txtD_z h(\gamma(t;t_0,z_0))]Z=:A(t;t_0,z_0)Z,
\ee
for $Z\in\R^d$, $A(t)=A(t;t_0,z_0)\in\R^{d\times d}$.
Let $\lambda_j=\lambda_j(t;t_0,z_0)$ for $j\in\{1,2,\ldots,d\}$ denote the eigenvalues of
$A(t)$. One guess, how to balance attraction and repulsion
near $\cM(t)$ is to consider
\be
\label{eq:defFlam}
\cF_{\lambda_j}(t):=\int_{t_0}^t\Re(\lambda_j(s))~\txtd s,
\ee
where $\Re(\cdot)$ denotes the real part of a complex number.
The problem with just monitoring eigenvalues is that they do not take into
account direction, i.e., we just get undirected rates of growth and decay.
Furthermore, the eigenvalues for fixed times do usually not imply any
stability statements for non-autonomous dynamical systems. Therefore, we
expect that the direct use of eigenvalues is insufficient in some many
cases and we shall demonstrate these issues in Section~\ref{ssec:solidbody}.

\subsection{Fast Subsystem Eigenvalues}
\label{ssec:fastslow}

Another idea is to consider multiple time scale dynamics and use fast subsystem
eigenvalues~\cite{KuehnBook}. Suppose~\eqref{eq:ODE} can be written in the standard fast-slow form
\be
\label{eq:fs1}
\begin{array}{rcl}
\varepsilon \frac{\txtd x}{\txtd \tau}  &=& \tilde{f}(x,\tilde{y},\tau,\varepsilon), \\
\frac{\txtd \tilde{y}}{\txtd \tau}  &=& g(x,\tilde{y},\tau,\varepsilon),
\end{array}
\ee
where $(x,\tilde{y})\in\R^{m+\tilde{n}}$, the maps $f:\R^{m+\tilde{n}+2}\ra \R^m$
and $g:\R^{m+\tilde{n}+2}\ra \R^m$ are smooth (we again require
$C^\I$ without further notice), let $h=(f,g)^\top$, and $0<\varepsilon\ll 1$.
Note that we can make the system autonomous by setting $\frac{\txtd \tau}{\txtd s}=1$.
This is equivalent to appending another slow variable $\dot{\tilde{y}}_{\tilde{n}+1}=1$,
so we shall now restrict to
\be
\label{eq:fs2}
\begin{array}{rcrcl}
\varepsilon \frac{\txtd x}{\txtd s} &=& \varepsilon \dot{x} &=& f(x,y,\varepsilon), \\
\frac{\txtd y}{\txtd s} &=& \dot{y} &=& g(x,y,\varepsilon),
\end{array}
\ee
where $z:=(x,y)\in\R^{m+n}$ for $n=\tilde{n}+1$. The critical manifold of~\eqref{eq:fs2}
is given by
\be
\cC_0:=\{(x,y)\in\R^{m+n}:f(x,y,0)=0\}.
\ee
Suppose $\cC_0$ is a smooth manifold which coincides with $\cM(t)$ on the given
domain for $t$, i.e., for $y_n$. Recall that $\cC_0$ is called normally hyperbolic
if the eigenvalues of the matrix $\txtD_x f(p,0)\in\R^{m\times m}$ have no zero
real part~\cite{KuehnBook}. Letting $\varepsilon\ra 0$ in~\eqref{eq:fs2} yields the slow 
subsystem
\be
\label{eq:ss}
\begin{array}{rcl}
0 &=& f(x,y,0), \\
\frac{\txtd y}{\txtd s} &=& g(x,y,0),
\end{array}
\ee
for the dynamics on the critical manifold. As before, consider
a trajectory $\gamma(s)=\gamma(s;s_0,z_0)$ contained in $\cC_0$. Then we can consider
the eigenvalues $\rho_j(s)$ for $j\in\{1,2,\ldots,m\}$ of the matrix $\txtD_x f(\gamma(s),0)$
and define
\be
\label{eq:defFrho}
\cF_{\rho_j}(t):=\int_{t_0}^t\Re(\rho_j(s))~\txtd s
\ee
to study the local attraction and repulsion rates near $\cC_0$. Note that there 
is still a choice of the neighbourhood $\cE$ of $\cC_0$, which can be selected to
be of size $\cO(\varepsilon)$ according to Fenichel's Theorem~\cite{Jones}
near the normally hyperbolic parts and has to be adapted to scalings of fast
subsystem bifurcation points; see~\cite[Ch.7-8]{KuehnBook}.

The functions $\cF_{\rho_j}$ are natural generalizations of the classical entry--exit maps
developed first in the context of delayed Hopf bifurcation~\cite{Neishtadt1,Neishtadt2};
in fact, the delayed Hopf case may also have buffer points so that the entry-exit
map does not give the correct exit point by itself. Since we do not consider any
oscillatory instabilities in this paper, we will not be further concerned with this
problem but see Section~\ref{sec:outlook}. However, since rigorous proofs for 
branch points are available~\cite{Schecter} (and references in~\cite{KuehnBook}), it
is important to have the fast-slow approach as a comparative benchmark. Furthermore, 
fast-slow techniques have been successfully applied to fluid dynamics of particles
(see~e.g.~\cite{RubinJonesMaxey}). The potential disadvantage of the fast-slow construction 
is that the ODE~\eqref{eq:ODE} is usually not directly available in the form~\eqref{eq:fs2}. 
This generically requires identifying a new parameter $\varepsilon$ and moving the invariant 
manifold $\cM(t)$ into a form where it becomes a slow manifold. In addition, not every
system has a well-defined time-scale separation.

\subsection{Finite-Time Lyapunov Exponent(s)}
\label{ssec:FTLE}

Another classic concept to deal with spectral properties of dynamical systems
defined on finite-time intervals are finite-time Lyapunov exponents (FTLEs); see
e.g.~\cite{LekienRoss,Haller5} for introductions. The intuition of FTLEs is to
consider a perturbation $\zeta\in\R^d$ to an initial condition for the flow
\be
\label{eq:FTLEidea}
\phi_{t_0}^{t}(z_0+\zeta)=\phi_{t_0}^{t}(z_0)+[\txtD\phi_{t_0}^{t}]\zeta
+\cO(\|\zeta\|^2),\qquad \text{as $\|\zeta\|\ra 0$,}
\ee
where $\|\cdot\|$ always denotes the Euclidean norm, the linearized flow map is
\be
\txtD \phi_{t_0}^t:\txtT\R^d\ra \txtT \R^d,\quad
(p,v)\mapsto (\phi_{t_0}^t(p),[\txtD_{z_0}z(t;t_0,p)]v),
\ee
and $\txtT \R^d\simeq \R^d$ denotes the tangent bundle to $\R^d$. Hence, the
expansion~\eqref{eq:FTLEidea} shows that the linearized flow map characterizes
local separation and contraction properties. Then one may define the maximum FTLE
as
\be
l_{\max}(t;t_0,z_0)=\frac{1}{|t-t_0|}\ln\left(
\max_{\|\zeta\| \neq 0}\frac{\|[\txtD \phi_{t_0}^{t}]\zeta\|}{\|\zeta\|}\right).
\ee
In addition, one may also consider FTLEs associated to specific directions. This
construction is conveniently expressed by considering the fundamental solution
$\Phi(t;t_0,z_0)\in\R^{d\times d}$ for $Z'=A(t)Z$ and letting $\delta_j=\delta_j(t;t_0,z_0)$
denote the singular values of $\Phi(t;t_0,z_0)$. Then one may define the associated
FTLEs~\cite{DoanKarraschNguyenSiegmund,ShaddenLekienMarsden} as
\be
l_j=l_j(t;t_0;z_0):=\frac{1}{t-t_0}\ln \delta_j(t;t_0,z_0).
\ee
Note carefully that FTLEs are dependent upon the final time $t$ and do not
constitute an infinitesimal notion. Since FTLE are defined over two time points,
$t_0$ and $t$, one possibility would be to just define balance functions by
\be
\cF_{l_j}(t):=l_j(t;t_0;z_0).
\ee
However, as the definition already suggests, the computation of FTLEs is
frequently impossible analytically and very challenging
numerically. Furthermore, FTLEs do not directly take into account the
existence of the invariant manifold $\cM(t)$ as the calculation works for any trajectory
in phase space via the variational equation.

\subsection{NILE Exponent}
\label{ssec:NILE}

Here we briefly recall the theory of normal infinitesimal Lyapunov exponents (NILE)
from~\cite{HallerSapsis1} in the context of~\eqref{eq:ODE}. Let
\be
\Pi_p^t:\txtT_p \R^d=\txtT_p\cM(t)\oplus \txtN_p\cM(t)\ra \txtN_p\cM(t),
\qquad \Pi_p^t(u,v)=v
\ee
denote the natural projection onto the normal space at each $p\in\cM(t)$. Then the
NILE at $p\in\cM(t)$ is defined by
\be
\label{eq:defNILE}
\sigma(p;t):=\lim_{s\ra 0^+}\frac1s \ln\left\|\left.\Pi^{t+s}_{\phi^{t+s}_t(p)}
[\txtD \phi_t^{t+s}]\right|_{\txtN_p\cM(t)}\right\|.
\ee
Although the definition may look complicated, $\sigma(p;t)$ is just the
infinitesimal growth rate in the normal direction to $\cM(t)$ at a point $p$.
Essentially, $\sigma(p;t)$ was designed in~\cite{HallerSapsis1} to be a more
computable measure of growth and decay rates for invariant manifolds in comparison
to the classical Lyapunov-type numbers~\cite{Fenichel1}. In particular, for a
smooth unit normal vector field $n(p;t)$ to $\cM(t)$ one has
\be
\label{eq:compn}
\sigma(p;t)=\max_{n(p,t)\in\txtN_p \cM(t)}n(p,t)^\top [\txtD_z h(p,t)]n(p,t)
\ee
by~\cite[Thm.4]{HallerSapsis1}. Of course, one could use~\eqref{eq:compn} as a
\emph{definition} instead of~\eqref{eq:defNILE}. In the
case when $\cM(t)$ can be written as a graph, one may consider $z=(x,y)\in\R^{m+n}$
and the ODE
\be
\begin{array}{lcl}
x'(t) &=& f(x(t),y(t),t),\\
y'(t) &=& g(x(t),y(t),t).
\end{array}
\ee
Suppose $\cM(t)=\{x=m_\cM(y,t)\}$ and define
\be
\Gamma(y,t):= \frac{\partial f}{\partial x}(m(y,t),y,t)-
\frac{\partial m_\cM}{\partial y}(y,t)~\frac{\partial g}{\partial x}(m(y,t),y,t),
\ee
then one may conclude~\cite[Thm.4]{HallerSapsis1} that
\be
\label{eq:NILE_mat}
\sigma(z,t)=\lambda_{\max} [\Gamma(y,t)+\Gamma(y,t)^\top]/2,
\ee
where $\lambda_{\max}[\cdot]$ is the largest eigenvalue of a symmetric matrix.
Considering the fact that $\sigma(p,t)<0$ corresponds to an attracting manifold
$\cM(t)$ while $\sigma(p,t)>0$ corresponds to a repelling one, we introduce the
definition
\be
\label{eq:Fsigma}
\cF_\sigma(t):=\int_{t_0}^t \sigma(\gamma(s;t_0,z_0),s)~\txtd s.
\ee
The advantage of NILE is that it does measure the generic (i.e., typical) maximum
growth and minimum decay rates for repelling, respectively attracting, manifolds. Furthermore,
it does take into account the geometry by focusing on the normal direction to the manifold
$\cM(t)$, which is the one relevant for entry and exit of neighbourhoods of $\cM(t)$. We 
again remark that we have to select the neighbourhood $\cE(t)$ so that the linearization
approximation in~\eqref{eq:Fsigma} is sufficiently accurate. In comparison to the 
fast-slow definition~\eqref{eq:defFrho} where natural scales in $\varepsilon$ appear, the
scales for NILE are related to uniform bounds on the normal directions and the requirement
that the leading linear normal direction is not dominated by nonlinear terms in $\cE(t)$. As
a practical strategy it seems safest to start with a neighbourhood $\cE(t)$ of very small
volume, where the approximation must be valid and gradually increase the neighbourhood. 
In particular, one may extend it until the linear approximation differs 
on several small compact test subsets from the full flow for a given tolerance.\medskip

Equation~\eqref{eq:NILE_mat} in conjunction with~\cite{DoanKarraschNguyenSiegmund} shows 
that the NILE concept is related to D-hyperbolicity~\cite{BergerDoanSiegmund1}, i.e., the 
generalization of the existence of an exponential dichotomy for finite-time intervals. 
Thereby, NILE also relates to other notions of spectra and hyperbolicity for non-autonomous 
systems on finite time intervals as discussed in~\cite{Berger,DoanKarraschNguyenSiegmund,Karrasch1}. 
In this context one usually obtains spectral intervals~\cite{Rasmussen1,Karrasch1}
and one could aim to use these intervals to define new balance functions.
However, for our purposes we aim to find a concept, which is analytically and conceptually
simple in low-dimensional examples and stays numerically computable for more complicated
flows. We shall see in several examples that a balance function based on NILE satisfies
these requirements quite well.

\subsection{General Balance Functions}
\label{ssec:bf}

We develop some basic general theory for the class of balance functions $\cF$ defined
above. Each function depends not only on its argument $T\in\R$ but also on the inital
time $t_0$, and the initial point $z_0$. Therefore, we also consider
\be
\cG_\upsilon:\cI\times \cI \times \cK \ra \R,\qquad \cG_\upsilon(t,t_0,z_0)=\cF_\upsilon(t)
\ee
using the $\cG$-notation to emphasize this dependence and with
$\upsilon\in\{\lambda_j,\rho_j,l_j,\sigma\}$; recall that $\lambda_j$ refers to 
instantaneous eigenvalues, $\rho_j$ to the fast-slow case, $l_j$ to finite-time Lyapunov
exponents (FTLEs), and $\sigma$ to the normal infinitesimal Lyapunov exponent (NILE). \medskip

\begin{prop}(basic properties)
\label{prop:basic}
 The following results hold:
 \begin{itemize}
  \item $\cG_{\upsilon}(t_0,t_0,z_0)=0$, for $\upsilon\in\{\lambda_j,\rho_j,\sigma\}$;
  \item $\lim_{t\ra t_0} \cG_{l_j}(t,t_0,z_0)$ exists;
  \item $\cG_\upsilon$ is bounded for $\upsilon\in\{\lambda_j,\rho_j,\sigma,l_j\}$.
 \end{itemize}
\end{prop}

\begin{proof}
The first statement is trivial due to the integral definition of the balance
function for $\{\lambda_j,\rho_j,\sigma\}$. For the second statement, it suffices
to study $t_0=0$ and the remaining cases will follow by a shift. Since $\gamma\in C^1$
it follows that $A(t)$ is $C^1$ in $t$ so $A(t)=A_0+tA_1+o(t)$ by Taylor's Theorem.
For the fundamental matrix we have
\be
\Phi(t,0,z_0)=\txte^{\int_0^t A(s)~\txtd s}=\txte^{t[A_0+o(t)]}.
\ee
So the singular values $\delta_j$ of $\Phi(t,0,z_0)$ satisfy
$\delta_j=\txte^{t[a_0+o(1)]}$ and the second result follows. The last
statement can be deduced from $\gamma\in C^1$; for example, consider
$\upsilon=\lambda_j$, then $A(t)$ is $C^1$ and defined on the compact set $\cI$
so the eigenvalues are bounded on $\cI$, and so are their real parts. Integrating
a bounded function over a compact set again yields a bounded function. Boundedness
with respect to the initial point follows from the compactness of $\cK$. The
other cases are equally easy.
\end{proof}\medskip

We remark that the last proof only used continuous differentiability.
Recall that we are primarily interested in zeros of balance functions
$t\mapsto \cG(t,t_0,z_0)=\cF(t)$. The existence of zeros is a global dynamical
problem and has to be considered on a case-by-case basis. However, the non-degeneracy
of a zero is a local property. Non-degeneracy of a nontrivial zero $T> t_0$
yields a transversality condition for the motion near the exit point, i.e., we
expect a true exit and not just sliding near the exit boundary. Therefore,
we have to study differentiability properties of $\cG$, particularly with respect
to $t$. To state the next result, we say that a family of matrices is NIC
(no-infinite-contacts)~\cite{Rainer} if none of its eigenvalues meet with an
infinite order of contacts when $t$ is varied, i.e., roots of the characteristic
polynomial have well-defined finite multiplicities for all $t$.\medskip

\begin{prop}(differentiability)
Consider $t_0,t$ in the interior of $\cI$ and $z_0\in\cK$. Then the following hold:
\begin{itemize}
 \item[(D1)] $t\mapsto \cG_{\lambda_j}(t;t_0,z_0)$ is $C^1$; if
 $A(s;t_0,z_0)$ is normal and NIC for all $s\in \cI$ then
 $(t,t_0,z_0)\mapsto\cG_{\lambda_j}(t,t_0,z_0)$ is $C^\I$;
 \item[(D2)] $t\mapsto \cG_{\rho_j}(t;t_0,z_0)$ is $C^1$; if
 $\txtD_xf(\gamma(s;t_0,z_0),0)$ is normal and NIC for all $s\in\cI$
 then $(t,t_0,z_0)\mapsto\cG_{\rho_j}(t,t_0,z_0)$ is $C^\I$;
 \item[(D3)] if $AA^\top$ is NIC then
 $(t,t_0,z_0)\mapsto \cG_{l_j}(t;t_0,z_0)$ is $C^\I$;
 \item[(D4)] suppose $\cM$ can be written globally as a graph of a
 $C^\I$-function, and $\Gamma(y,s)$ is NIC for all $s\in\cI$ then
 $(t,t_0,z_0)\mapsto \cG_{\sigma}(t;t_0,z_0)$ is $C^\I$;
\end{itemize}
\end{prop}

\begin{proof}
Regarding (D1), differentiability in $t$ is immediate from the Leibniz integral
formula, which yields a continuous derivative. For smoothness in the
initial data, consider first $t_0$. Note that $A=A(s;t_0,z_0)$ is a $C^\I$ function
of $t_0$ by assumption on the smoothness of the vector field and the assumption
that $\gamma\in C^\I$. Since $A$ is normal, NIC, and smooth,
it follows that its eigenvalues $\lambda_j=\lambda_j(s;t_0,z_0)$ are $C^\I$ as
functions of $t_0$~\cite[Thm.7.8]{Rainer}. Therefore, the corresponding real parts
are $C^\I$ as well. Using the Leibniz integral formula,
we obtain smoothness with respect to $t_0$. The same
argument can now be applied to each of the coordinates of $z_0$ and (D1) follows.
(D2) is immediate as the proof from (D1) carries over. For (D3), the
first observation is that $AA^\top$ is a normal matrix. Since $AA^\top$ is also NIC, 
the singular values are $C^\I$; now we can just use the same Taylor expansion argument 
as in the proof of Proposition~\ref{prop:basic} to get that
$l_j$ is $C^\I$. For (D4), note that since $\cM$ can be written as a $C^\I$-graph,
it follows that $\Gamma(y,t)$ is $C^\I$, so using formula~\eqref{eq:NILE_mat}
the last result follows by the same arguments as (D1)-(D2).
\end{proof}\medskip

The differentiability results can be improved but
selecting eigenvalue-type quantities smoothly is an extremely technical topic and
many special cases may occur~\cite{Rainer}. Here we are mainly interested in the
dependence of zeros of $\cF$ on the input data. A direct application
of the implicit function theorem yields:\medskip

\begin{cor}
\label{cor:sd}
Suppose $h\in C^\I$ and $\gamma\in C^\I$ for each trajectory and the assumptions
in (D1)-(D4) hold. Consider
$(t,t_0,z_0)\mapsto\cG_\upsilon(t,t_0,z_0)$ for
$\upsilon\in\{\lambda_j,\rho_j,\sigma,l_j\}$ and suppose $t\mapsto \cF_\upsilon(t)$
has an non-degenerate root at $T$, i.e., $\cF'(T)\neq 0$. Then there
exists a neighbourhood of $(T,t_0,z_0)$ and a locally unique continuous one-dimensional
curve of solutions to $0=\cG_\upsilon(t,t_0,z_0)$.
\end{cor}\medskip

The last result is not really surprising. It just states that if we find a zero
of the balance function, and this zero is isolated, then the $C^1$-dependence of
the balance function on its input data guarantees that the zero perturbs uniquely.
However, for applications it is of paramount importance to have at least some
computable conditions for robustness such as $\cF'(T)\neq 0$.\medskip

\textbf{Remark:} \textit{Regularity results are also relevant to study the dependence upon initial 
conditions such as sets of the form
\be
\label{eq:St}
\cS_t := \{ z_0 \in \mathcal{K} : \mathcal{F}(t) = 0\},
\ee
i.e., considering the evolution of co-dimension one submanifolds of initial conditions. 
Fixing a neighbourhood size of $\cM(t)$, makes $\cS_t$ depend only upon $z_0$. However, if 
$\cM(t)$ interacts with another invariant manifold then an initially smooth connected set
$\cS_t$ may break up. Furthermore, since $\cS_t$ are sets imposing balance conditions, it
would be natural ask, how this relates to Lagrangian coherent structures and we aim to
pursue this analysis in future work.}

\subsection{Perturbation Expansion}
\label{ssec:perturb}

Having reduced the entry-exit problem to a root-finding problem, we remark that it is now
possible to apply standard perturbation techniques if suitable boundedness and regularity
conditions are satisfied. For example, consider the case when we have to solve
\be
\label{eq:pert_main}
0=\cF(t;\delta)
\ee
for some small system parameter $\delta\geq 0$. If $\cF(T;0)=0$ and
$\frac{\partial \cF}{\partial t}(T;0)\neq 0$
then the implicit function theorem yields that we can locally solve~\eqref{eq:pert_main}
via $T=T(\delta)$ and formally can make the perturbation ansatz:
\be
T(\delta)=T_0+\delta T_1+\delta^2 T_2+\cdots,\qquad T_j\in\R,~0<\delta\ll1.
\ee
Furthermore, if~\eqref{eq:pert_main} depends upon a trajectory $\gamma(t)\subset \cM(t)$
which has a perturbation expansion
\be
\gamma(t)=\gamma_0(t)+\delta\gamma_1(t)+\delta^2\gamma_2(t)+\cdots
\ee
then we can just plug all terms into~\eqref{eq:pert_main} and formally expand in $\delta$,
collect terms of different orders, and aim to solve the problem perturbatively if the
perturbation problem is regular.

\section{Passive Models}
\label{sec:passive}

Here we shall consider two basic models that represent elementary
flows to benchmark the previously introduced concepts. 'Passive' refers to
the fact that particles inside these flows are viewed as passively transported
while Section~\ref{sec:creep} presents an 'active' case, where there is a
fully-coupled particle-fluid interaction.

\subsection{Regularized Solid-Body Rotation}
\label{ssec:solidbody}

We start with a basic example, which is nevertheless quite insightful as it is 
relatively easy from a computational perspective and demonstrates the basic features 
of balance functions. Consider the ODE
\be
\label{eq:rot1}
\begin{array}{lcl}
z_1'&=&-(z_2-\beta),\\
z_2'&=& z_1(1-\txte^{-\alpha z_2}),\\
\end{array}
\ee
where $(z_1,z_2)\in \R\times [0,+\I)$, and $\alpha,\beta>0$ are parameters.
An example of the phase portrait is shown in Figure~\ref{fig:02}. Note that the
center of rotation of the flow is at $(z_1,z_2)=(0,\beta)$. $\cM=\{y=0\}$
is a time-independent invariant manifold with dynamics $z_1'=\beta$. We are
interested in the entry-exit relation of a point particle getting passively
transported by the flow~\eqref{eq:rot1} in and out of a neighbourhood
\be
\cE(\chi):=\{(z_1,z_2)\in\R^2:z_2\in[0,\chi]\}.
\ee

\begin{figure}[htbp]
\psfrag{z1}{$z_1$}
\psfrag{z2}{$z_2$}
	\centering
		\includegraphics[width=0.7\textwidth]{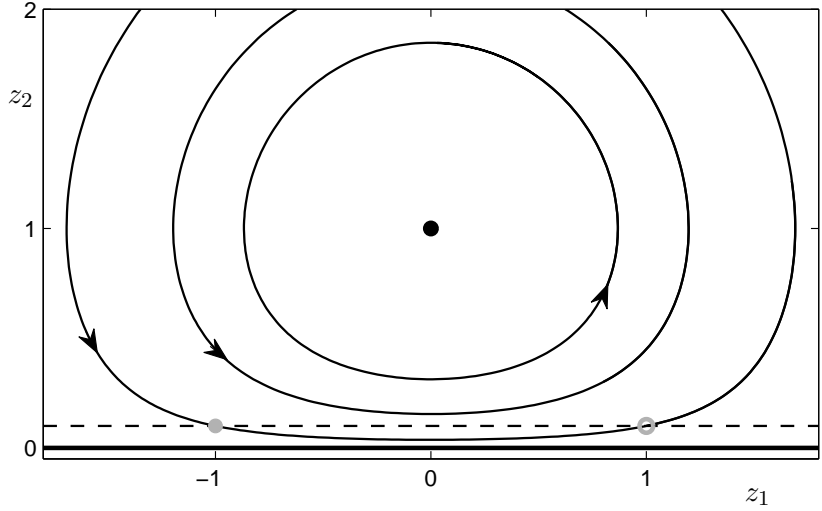}	
		\caption{\label{fig:02}Trajectories of the ODE~\eqref{eq:rot1} for parameters $\beta=1$, $\alpha=2$.
		The invariant manifold $\cM=\{z_2=0\}$ (thick black line) is a fixed wall with neighbourhood
		$\cE(\chi)$ (indicated by a dashed black line). The center of rotation at $(0,\beta)$ is marked by
		a black dot. A typical relationship between entry (grey dot) and
		exit (grey circle) of a particle transported by the flow is indicated as well.}
\end{figure}

Assume that $\chi>0$ is sufficiently small so we aim to approximate the
residence time of the particle from the linearized system near $\cM$. We expect
a symmetric entry-exit relation. Indeed, given the initial point
$z_0=(-b,\chi)\in\partial\cE(\chi)$, we know already that the first exit of
the point is given by $z_T=(b,\chi)$ due to the symmetry
\be
(z_1,z_2,t)\mapsto (-z_1,z_2,-t)
\ee
of~\eqref{eq:rot1}. We discard this fact temporarily and apply the different
functions $\cF$. We start by just naively considering the concept of
instantaneous eigenvalues discussed in Section~\ref{ssec:evalinst}.\medskip

\begin{prop}
\label{prop:bad_iev}
There exists an open set of parameters $\alpha,\beta,b>0$ such that the exit
computed by balancing $\cF_{\lambda_j}$ is correct, while there also exists
an open set of parameters where balancing $\cF_{\lambda_j}$ gives an incorrect
exit point.
\end{prop}

\begin{proof}
As a trajectory in $\cM$, we must take
\be
\label{eq:trajectrot}
\gamma(t;t_0,z_0)=(\beta(t-t_0)-b,0)^\top.
\ee
Without loss of generality consider $t_0=0$. Linearization yields that
\be
A=A(t;t_0,z_0)=\left(
\begin{array}{cc}
0 & -1 \\
1 & \alpha(\beta t-b)\\
\end{array}
\right).
\ee
The two eigenvalues of $A$ are
$\lambda_\pm=\frac12\left(\alpha(\beta t-b)\pm\sqrt{[\alpha(\beta t-b)]^2-4}\right)$.
Therefore, consider
\bea
\cF_{\lambda_\pm}(T)&=&\int_0^T \Re(\lambda_-(s))~\txtd s,\nonumber\\
&=& \int_0^{\frac{4+b\alpha}{\alpha \beta} \wedge T} \frac{\alpha}{2}(\beta s-b)~\txtd s\pm
\int_{\frac{4+b\alpha}{\alpha \beta} \wedge T}^T \frac12\left(\alpha(\beta t-b)-
\sqrt{[\alpha(\beta t-b)]^2-4}\right)~\txtd s, \label{eq:intF1}
\eea
where $\frac{4+b\alpha}{\alpha \beta} \wedge T =\min\{\frac{4+b\alpha}{\alpha \beta}, T\}$.
The second integral can be evaluated in certain cases but is a lengthy expression.
It is interesting to just consider certain cases. Suppose $b\alpha <4$, then
$\cF_{\lambda_\pm}$ has a zero at $T=\frac{2b}{\beta}$ since $\frac{2b}{\beta}<\frac{4+b\alpha}{\alpha \beta}$.
This yields the correct exit point since $\beta\frac{2b}{\beta}-b=b$. This proves the
first part of the proposition. If $b\alpha>4$ the second integral is relevant in~\eqref{eq:intF1}.
For example, take $\beta=1$, $\alpha=2$ and $b=3$, then we have
\be
\cF_{\lambda_\pm}(T)= -\frac52\pm \int_{5}^T t-3-\sqrt{[(t-3)]^2-1}~\txtd s
\ee
The last integral can be evaluated and one checks that $\cF_{\lambda_\pm}(6)\neq 0$.
However, $\gamma(6;0,(-3,0))=(6-3,0)=(3,0)$ is the correct exit point. A direct
application of Corollary~\ref{cor:sd} yields an open set of parameters where $\cF_{\lambda_j}$
yields an incorrect exit point.
\end{proof}\medskip

Since monitoring the real parts of instantaneous eigenvalues and requiring an integrated balance
function $\cF_{\lambda_j}$ does not yield the correct result in relevant cases, we discard this
option from now on. Another option considered in Section~\ref{ssec:fastslow} is to scale the problem
differently.\medskip

\begin{prop}
There exists a scaling of~\eqref{eq:rot1} converting it to the standard fast-slow form~\eqref{eq:fs2}.
Furthermore, balancing the function $\cF_\rho$ yields the correct exit point.
\end{prop}

\begin{proof}
Consider the scaling $(z_1,z_2,t)=(y/\sqrt\varepsilon,x,s/\sqrt\varepsilon)$
and observe that~\eqref{eq:rot1} then becomes a fast-slow system
\be
\label{eq:rot2}
\begin{array}{rcl}
\varepsilon \dot{x}&=& y(1-\txte^{-\alpha x}),\\
\dot{y}&=&-(x-\beta).\\
\end{array}
\ee
The critical manifold is $\cC_0=\{y=0\}\cup \{x=0\}$. Note that $\cC_0$ is
not a smooth manifold but the subset $\{x=0\}=\cM$ is an invariant manifold
for~\eqref{eq:rot2} for any $\varepsilon>0$ and we focus just on $\cM$ here.
The slow subsystem is given by $\dot{y}=\beta$ so $\gamma(s)=(0,\beta(s-s_0)-b)$.
In the notation of Section~\ref{ssec:fastslow} we have the linearization of
the fast vector field yields just one eigenvalue so that
\be
\rho=\txtD_xf(\gamma,0)=\alpha[\beta(s-s_0)-b].
\ee
Therefore, the balance function can be calculated, say for $s_0=0$,
\be
\label{eq:rotcalcfs1}
\cF_\rho(T)=\alpha \int_0^T \beta s-b~\txtd s =\alpha\left[\frac{\beta}{2}T^2-bT\right]
=\alpha T\left[\frac{\beta}{2} T-b\right].
\ee
Now the balance condition $\cF_\rho(T)=0$ provides the correct answer that
the exit point is given, via the zero $T=2b/\beta$, as $(x,y)=(0,b)$.
\end{proof}\medskip

Furthermore, we easily check that the zero of the balance function is indeed
isolated since $\cF_\rho'(2b)=\alpha(2b-b)=\alpha b>0$ since $\alpha,b>0$ by assumption.
However, the fast-slow calculation did require some insight how to scale the problem.
As the next option, we consider FTLE balance functions. In fact, the abstract framework
is relatively straightforward. Considering the non-autonomous linear problem
$Z'=A(T)Z$, we know that a fundamental solution can be written as
\be
\Phi(t;t_0,z_0)=\exp\left[\int_{t_0}^t A(r)~\txtd r\right].
\ee
Working out the integral is easy but the algebraic form of matrix exponential is
extremely lengthy in the general case. However, observe that
\be
\int_{t_0}^t A(r)~\txtd r=\left(
\begin{array}{cc}
 0 & -t \\
 t & \frac{1}{2} t (\beta t-2 b) \alpha  \\
\end{array}
\right)
\ee
So if $t=2b/\beta$ then we easily find that the singular values of $\Phi(2b/a,0,(-b,0))$
are both zero.
Therefore, it follows indeed that $\cF_{l_1}(2b/\beta)=0=\cF_{l_2}(2b/\beta)$ yielding the
correct balance time for this case.\medskip

\begin{prop}
Balancing $\cF_{l_j}$ yields the correct entry-exit relationship for~\eqref{eq:rot1}.
\end{prop}\medskip

The formulas for FTLE can be extremely cumbersome. Calculating the FTLEs analytically
for all $t$ corresponds actually to a full solution of the non-autonomous linear
system along the invariant manifold. Hence, it seems worthwhile to search for simpler
balance functions that take into account the existence of $\cM$ and its geometry.
A natural option seems to be the NILE as discussed in Section~\ref{ssec:NILE}.\medskip

\begin{prop}
Balancing $\cF_{\sigma}$ yields the correct entry-exit relationship for~\eqref{eq:rot1}.
\end{prop}

\begin{proof}
Consider~\eqref{eq:trajectrot} for $t_0=0$ and the invariant manifold $\cM=\{z_2=0\}$
for~\eqref{eq:rot1}. A unit normal vector field is given by $n(p,t)=(0,1)^\top$
and formula~\eqref{eq:compn} yields
\benn
\sigma(\gamma,t)=(0,1)
\left(
\begin{array}{cc}
0 & -1 \\
1 & \alpha(\beta t-b)\\
\end{array}
\right)
\left(
\begin{array}{c}
0  \\
1 \\
\end{array}
\right)=\alpha(\beta t-b).
\eenn
Hence, the balance function $\cF_\sigma$ also yields the correct result for the
exit point by the same calculation as in~\eqref{eq:rotcalcfs1}.
\end{proof}\medskip

No preliminary scaling by $\varepsilon$ was necessary for the NILE case nor was it
necessary to solve the full non-autonomous system. Hence, the concept seems to
be well-suited to take into account the geometry of the invariant manifold. Furthermore, we
may calculate eigenvalues directly from the linearized problem without
the need to scale by a small parameter. However, NILE is a
single rate so one should keep in mind that it would be useful to generalize
the case from a unit normal vector field to $\cM$ to other transverse frames
when projection operators depend upon time and position;
cf.~\cite[p.614-615]{HallerSapsis1}. However, generically the entry (resp.~departure)
speed is governed by the largest normal rates in the attracting (resp.~repelling) regime
so NILE already captures what we are interested in for most applications.

\subsection{The Kuhlmann--Muldoon model}
\label{ssec:MuldoonKuhlmann}

An axisymmetric liquid bridge of length $d$, stabilized by surface tension, can be 
established between two parallel and coaxial wetted disks of cylindrical rods of 
radius $R$. If the disks are heated differentially and the aspect ratio 
$\Gamma=d/R$ is of order $\cO(1)$, temperature gradients along the capillary interface 
induce a toroidal vortex via the thermocapillary effect~\cite{ScrivenSternling}, which can 
become modulated azimuthally for large imposed temperature differences. Muldoon and 
Kuhlmann~\cite{MuldoonKuhlmann} proposed a closed-form approximation to the three-dimensional 
time-dependent Navier--Stokes flow in a cylindrical liquid bridge aiming at modeling the 
rapid de-mixing of particles into curious accumulation structures found experimentally 
(see e.g.~\cite{SchwabeMizevUdhayasankarTanaka}); see also Section~\ref{sec:motivation} for 
further background. Using a separation ansatz and free-slip 
boundary conditions on the supporting disks, the axisymmetric part of the flow field was 
represented by harmonic functions. The algebraic radial dependence of the axisymmetric flow 
is then dictated by the incompressibility constraint. The remaining integration constants 
were used to fit the strength and the radial shape of the vortex to the numerically obtained 
Navier--Stokes solution. The axisymmetric steady part of the Kuhlmann--Muldoon model flow 
is symmetric with respect to a midplane $\{z_2=0\}$. Here we  extend this flow by an antisymmetric 
term which may take care of the inertia-induced asymmetry of the real flow. This also
aims to test balance functions in a setup without symmetry. Let the fluid 
be confined to the cylindrical volume $(z_1,\phi,z_2)\in[0,1/\Gamma]\times [0,2\pi]
\times [-0.5,0.5]$ and discard the second component. Then we define the planar 
model flow for an aspect ratio $\Gamma=2/3$ by
\be
\label{eq:MK_ODE}
\begin{array}{lcl}
   z_1' &=& \pi z_1^\eta(1-\frac23 z_1)[\sin(\pi z_2)-2\alpha_{\txts}\cos(2\pi z_2)],\\
   z_2' &=& [(\eta+1)z_1^{\eta-1}-
\frac23(\eta+2)z_1^\eta][\cos (\pi z_2)+\alpha_{\txts}\sin(2\pi z_2)],
\end{array}
\ee
where $z_1\in(0,\frac32]$ and $z_2\in(-\frac12,\frac12)$. The parameter $\alpha_{\txts}$ controls the 
asymmetry of the flow, and the exponent $\eta\in[4,5]$ determines the radial shape of the vortex. 
Figure~\ref{fig:model_flow} shows examples of the flow field.

\begin{figure}[htbp]
\small
\psfrag{r}{$z_1$}
\psfrag{z}{\rotatebox{-90}{$z_2$}\ }
\psfrag{0}{$0$}
\psfrag{0.5}[br][br]{$0.5$}
\psfrag{1}{$1$}
\psfrag{1.5}{$1.5$}
\psfrag{-0.5}[br][br]{$-0.5$}
    \centering
\subfigure[]{\includegraphics[clip,width=0.48\textwidth]{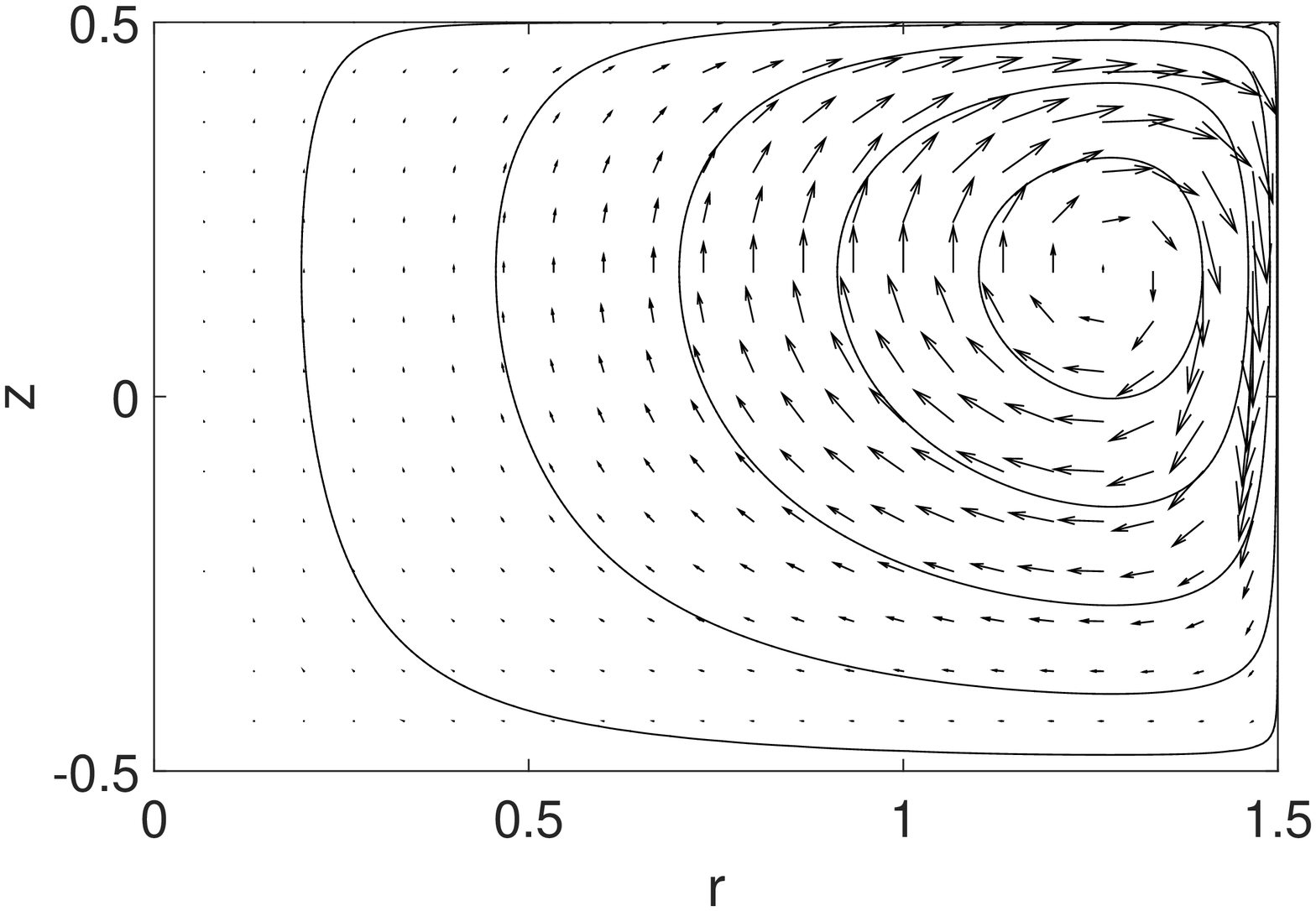}}
        \hfill
\subfigure[]{\includegraphics[clip,width=0.48\textwidth]{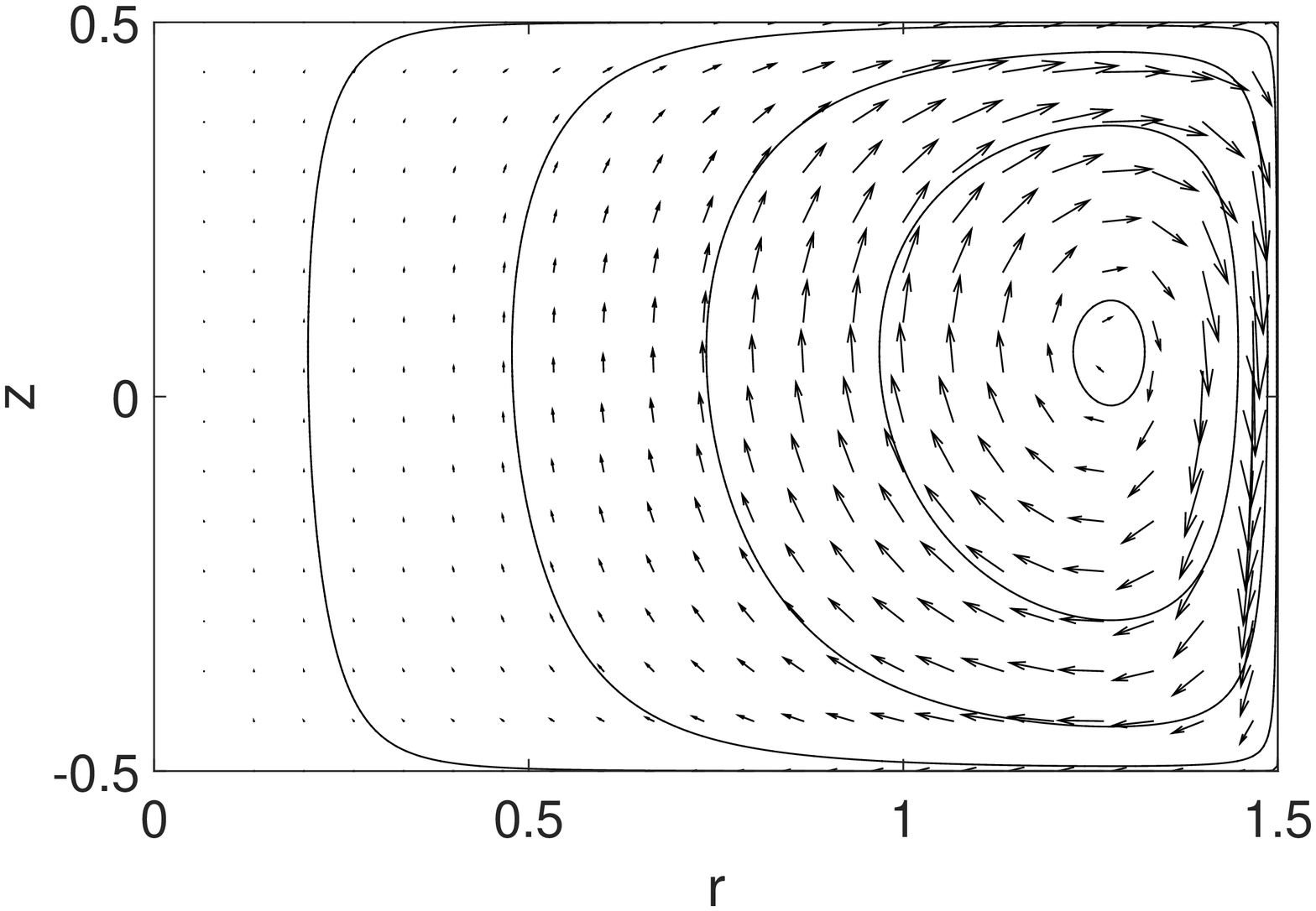}}
        \caption{\label{fig:model_flow} Streamlines and vector field of the model flow \eqref{eq:MK_ODE} for $\eta=4.74$ and $\alpha_s = 0.5$ (a) and $\alpha_s=0.1$ (b).}
\end{figure}

The relevant invariant manifold we are interested in is $\cM=\{z_1=\frac32,z_2\in(-\frac12,\frac12)\}$ 
(capillary surface). Note that the flow on $\cM$ can be calculated from
\be
\label{eq:ODE_MK_onM}
z_2'=-\left(\frac32\right)^{\eta-1} [\cos (\pi z_2)+\alpha_{\txts}\sin(2\pi z_2)].
\ee

Unfortunately, there does not seem to be a closed form solution to the
ODE~\eqref{eq:ODE_MK_onM}. As before, let $\gamma$ denote a trajectory
contained inside $\cM$. The linearization $A(t)=\txtD h(\gamma(t))$ with
$t_0=0$ starting from initial data $(z_1,z_2)=(\frac23-\xi,b)$ along $\gamma$
can be calculated
\benn
\left(
\begin{array}{cc}
 \left(\frac{3}{2}\right)^{\eta-1} (2 \alpha_\txts \cos (2 \pi  \gamma_2)-\sin (\pi  \gamma_2)) & 0 \\
 -\left(\frac{2}{3}\right)^{2-\eta} (2 \eta+1) (\cos (\pi  \gamma_2)+A \sin (2 \pi  \gamma_2)) &
 -\left(\frac{2}{3}\right)^{1-\eta} \pi  (2 \alpha_\txts \cos (2 \pi  \gamma_2)-\sin (\pi  \gamma_2)) \\
\end{array}
\right)
\eenn
The full analytical computation of FTLEs is not possible, hence we start by
considering NILE. Since a unit normal vector field to $\cM$ is just given by
$n(p,t)=(0,1)^\top$, we can easily calculate
\be
(0,1)~A~\left(\begin{array}{c} 0\\ 1\end{array}\right)=
-\left(\frac{2}{3}\right)^{1-\eta} \pi  (2 \alpha_\txts \cos (2 \pi  \gamma_2)-\sin (\pi  \gamma_2))
\ee
Now one can actually numerically integrate~\eqref{eq:ODE_MK_onM} and then
insert the solution $\gamma$ into the integral
\be
\label{eq:intMK}
\cF_\sigma(t)=-\left(\frac{2}{3}\right)^{1-\eta} \pi   \int_0^t
2 \alpha_\txts \cos (2 \pi  \gamma_2)-\sin (\pi  \gamma_2)~\txtd s.
\ee

\begin{figure}[htbp]
\psfrag{T}{$T$}
\psfrag{alpha}{$\alpha_\txts$}
\psfrag{eta}{$\eta$}
\psfrag{a}{(a)}
\psfrag{b}{(b)}
	\centering
		\includegraphics[width=0.75\textwidth]{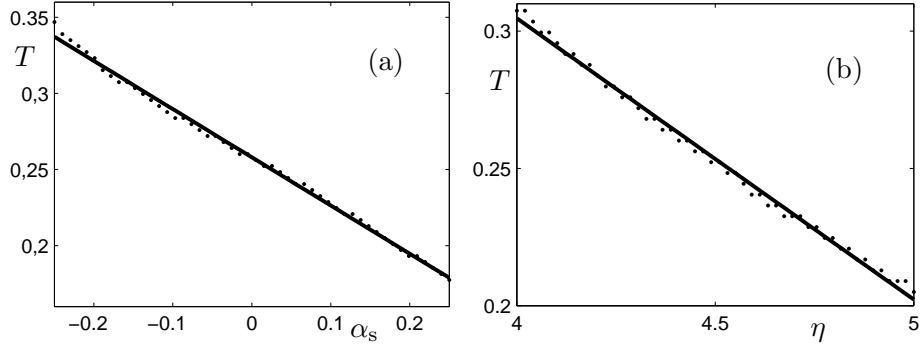}	
		\caption{\label{fig:03}Computation for the Kuhlmann--Muldoon model relating the initial
		time $t_0=0$ to a final exit time $T$ using the integral formula~\eqref{eq:intMK}
		finding $T$ by requiring $\cF_\sigma(T)=0$.
		(a) The parameter $\eta=4.74$ and the initial condition
		$(\frac32,\frac{2}{5})$ are fixed and the asymmetry parameter is varied. The computed
		exit times are marked as dots; a linear fit to the computed points is shown. (b) Identical
		setup as in (a) except that the parameter $\alpha_\txts=\frac{1}{10}$ is fixed and the parameter
		$\eta$ is varied to compute the exit time.}
\end{figure}

We already see that if we are just interested in zeros of $\cF_\sigma$
we may discard the prefactor in~\eqref{eq:intMK} as $\eta\in[4,5]$. Figure~\ref{fig:03}
shows a computation based upon the integral formula~\eqref{eq:intMK} to study the
influence of the two parameters $\alpha_\txts,\eta$ on the exit time $T$. We observe a
very clear linear relationship between the parameters and the exit time $T$. The initial
condition has been fixed to $(\frac32,\frac{1}{10})$. If $\alpha_\txts,\eta$ are decreased,
we observe that this induces a quicker escape in terms of escape times. Figure~\ref{fig:04}
converts the escape times into escape points by integrating the one-dimensional
ODE~\eqref{eq:ODE_MK_onM}. Interestingly the dependence on the two parameters now shows
a completely different behaviour. For the asymmetry parameter $\alpha_\txts$, a nonlinear
behaviour is observed while for $\eta$, we observe a similar exit point regardless of the
parameter, i.e., the changing escape time is compensated by a different speed on $\cM$.

\begin{figure}[htbp]
\psfrag{yT}{\scriptsize{$z_2(T)$}}
\psfrag{alpha}{$\alpha_\txts$}
\psfrag{eta}{$\eta$}
\psfrag{a}{(a)}
\psfrag{b}{(b)}
	\centering
		\includegraphics[width=0.95\textwidth]{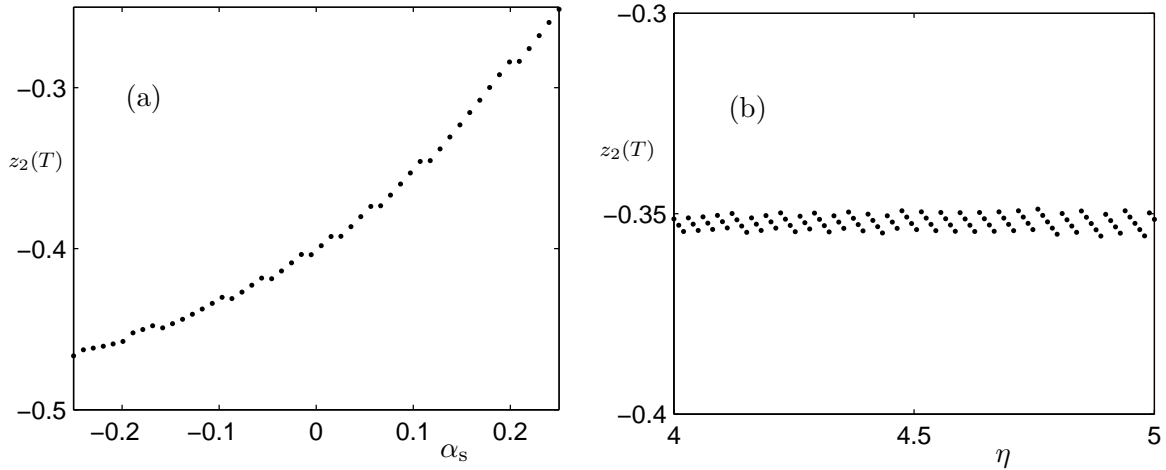}	
		\caption{\label{fig:04}Computation for the Kuhlmann--Muldoon model relating the initial
		time $t_0=0$ to a final exit time $T$ using the integral formula~\eqref{eq:intMK}. In
		contrast to Figure~\ref{fig:03} we show the exit point $z_2(T)$ on the vertical axis
		obtained from numerically integrating~\eqref{eq:ODE_MK_onM} up to time $T$.
		(a) The parameter $\eta=4.74$ and the initial condition
		$(\frac32,\frac{2}{5})$ are fixed and the asymmetry parameter is varied. The computed
		exit points are marked as dots. (b) Identical setup as in (a) except that the 
		parameter $\alpha_\txts=\frac{1}{10}$ is fixed and the parameter $\eta$ is varied to 
		compute the exit point.}
\end{figure}

The analysis shows that it is very convenient to use balance functions numerically and
to determine the influence of different parameters on particle motion near $\cM$.\medskip

\textbf{Remark:} Another way to arrive at the entry-exit map is to scale the original
ODE~\eqref{eq:MK_ODE} by first moving $\cM$ to the manifold $\{z_1=0,z_2\in(-0.5,0.5)\}$
and then scaling $z_1$ by a suitable power of a small parameter $\varepsilon$ to make
the $z_1$-variable a fast $x$-variable. However, the problem already shows that scaling
becomes more and more involved once the manifold $\cM$ is nontrivial.

\section{Numerical Simulations for a Particle near a Free-Surface}
\label{sec:creep}

Particle-laden flows are a class of multiphase flows in which there is a continuously
connected fluid-phase and a dispersed particle phase which do not mix. When the particles
scales are very small compared to the fluid scales, neglecting the feedback effect of
the particles on the fluid flow can be a good approximation in numerically calculating
each particle trajectory. In this case we talk about one-way coupling simulations. However,
if the particle passes very close to a wall/free-surface, the particle and the lubrication-gap 
scales must be solved and the feedback effect of the particle on the fluid phase is
essential~\cite{RomanoKuhlmann,KosinskiKosinskaHoffmann}, i.e., the non-autonomous dynamics
really differs from the autonomous case. Simulating both phases can be 
very expensive computationally. Therefore, understanding the physical mechanisms 
which play the main role in the particle--boundary interaction contributes to an accurate 
modelling of this phenomenon. Taking into account lubrication effects can lead to a significant 
improvement in predicting the particle trajectories. Hence, this setup is an excellent test
case for the framework of balance functions presented above.

\begin{figure}
\small
\psfrag{tau0}[lc][lc]{$\tau_0$}%
\psfrag{y1}[lc][lc]{$z_2$}%
\psfrag{x1}[lc][lc]{$z_1$}%
\psfrag{L1}[lc][lc]{$L$}%
\psfrag{a0}[lc][lc]{$2a$}%
\begin{center}
\includegraphics[clip,height=0.35\textwidth]{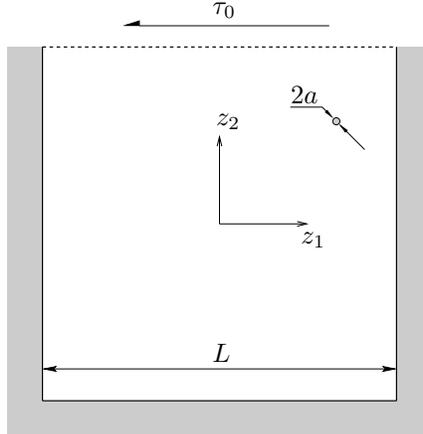}
\caption{Sketch of the two-dimensional shear--stress-driven cavity seeded with a particle
of radius $a = 0.01\cdot L$.}
\label{fig:06}
\end{center}
\end{figure}

Consider a two-dimensional square cavity filled with an incompressible Newtonian 
liquid of density $\rho_\text{f}$ and kinematic viscosity $\nu$ as shown in 
Figure~\ref{fig:06}; we denote the horizontal and vertical coordinates
by $z_1$ and $z_2$ respectively similar to our notation above and set $z=(z_1,z_2)^\top$. 
The cavity of linear length $L$ is open
from above and the fluid flow is driven by a constant shear stress, $\tau_0$, in
$z_1$-direction. In the limit of asymptotically large surface tension the interface 
is flat. The flow in the shear--stress-driven cavity can be modelled 
using the Navier--Stokes equations and employing a viscous scaling
\begin{align}
   {\hat{u}} = \frac{\nu}{L}{u} , \quad {\hat{z}} =  L {x}, \quad \hat{t} 
   = \frac{L^2}{\nu} t , \quad \hat{p} = \frac{\rho_\text{f}\nu^2}{L^2} p ,
\end{align}
where 'hats' indicate the dimensional variables and 'no-hats' the non-dimensional ones, one
obtains the non-dimensional Navier--Stokes system
\begin{subequations}\label{eq:NavierStokes_ShearStress}
\begin{align}
\nabla \cdot {u} &= 0,\\
\left(\partial_t +  {u} \cdot \nabla \right){u} &= - \nabla p + \nabla^2 {u},
\end{align}
\end{subequations}
where ${u}=u(z,t)\in\R^2$ and $p=p(z,t)\in\R$ represent the flow velocity and 
pressure fields, respectively. The flow field must satisfy the no-slip and constant-stress 
boundary conditions
\begin{subequations}\label{eq:NavierStokesBC1_ShearStress}
\begin{align}
\text{on }\{z_1 = \pm 1/2\},& \quad u_1 = 0 = u_2; \\
\text{on }\{z_2 = - 1/2\},& \quad u_1 = 0 = u_2; \\
\text{on }\{z_2 = 1/2\},& \quad \partial_{z_2} u_1 = -\Rey, \quad u_2 = 0;
\end{align}
\end{subequations}
where $\Rey$ is the Reynolds number
\begin{equation}
\Rey = \cfrac{\tau_0 L^2}{\rho_\text{f}\nu^2}.
\end{equation}
and in the following it will be set equal to 1000. Considering now the cavity 
seeded with one rigid circular particle of radius $a$, rigid body motion equations 
will define its trajectory in the fluid flow
\begin{subequations}\label{eq:RigidBodyMotion}
\begin{align}
  F &= M~\frac{\txtd V}{\txtd t},\\
  \cT &= I~\frac{\txtd \Omega}{\txtd t}
\end{align}
\end{subequations}
where $F$ and $\cT$ are the force and torque acting on the particle, $V$
and $\Omega$ denote its translational and rotational velocities, and $M$ and $I$ are the mass
and the inertia tensor of the particle.
The coupling between the particle and the fluid motions results from the no-slip and
no-penetration conditions on the surface of the particle. In the following the case of
particle radius $a = 0.01 L$ is considered for a particle-to-fluid density ratio
$\varrho = \rho_\text{p}/\rho_\text{f} = 2$. The initial particle position is set to 
$(z_\text{$1$,p} , z_\text{$2$,p}) = (0.38 , 0.2)$. The equations are discretized on a 
static grid via a discontinuous-Galerkin finite-element-method (DG-FEM) coupling it with 
the so-called smoothed-profile method (SPM). The details of the numerical method are
discussed in~\cite{RomanoKuhlmann1} with further background in~\cite{LuoMaxeyKarniadakis,NakayamaYamamoto,KarniadakisIsraeliOrszag,HesthavenWarburton}. 
Here we just note that we use in space a discontinuous Galerkin--finite-element 
method with triangular grid with warp-blend-nodes. Polynomial elements of order five 
are employed for the numerical simulation with triangles with side length $4 \times 10^{-3}$ 
in the $z_1$-direction. Along the $z_2$-direction a stretching function proportional 
to $z_2^(0.4)$ is adopted with minimum side length of $9\times 10^{-4}$ for the elements 
in contact with the free-surface. The time step is fixed to $1 \times 10^{-7}$~\cite{RomanoKuhlmann1}.

\begin{figure}[htbp]
\psfrag{Dz}{\scriptsize{$\cF_\sigma$}}
\psfrag{V}{\scriptsize{$\cF_v$}}
\psfrag{x}{\scriptsize{$z_1$}}
\psfrag{y}{\scriptsize{$z_2$}}
\psfrag{t}{\scriptsize{$t$}}
\psfrag{NILE predictor}{NILE balance}
\psfrag{Nonlinear predictor}{nonlinear balance}
\centering
		\includegraphics[width=0.9\textwidth]{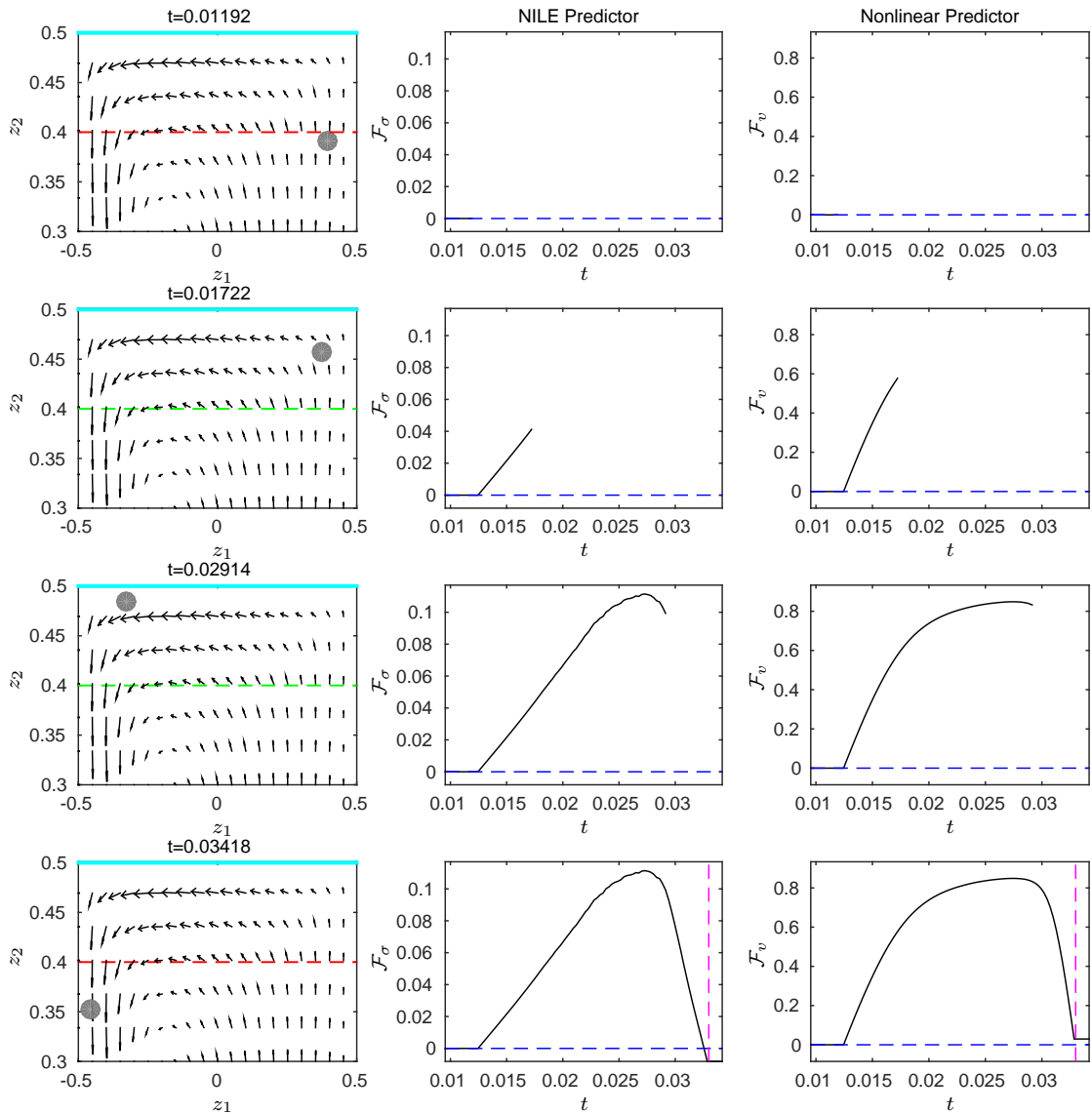}	
		\caption{\label{fig:05}Each row of the figure corresponds to one time snapshot of the 
		full numerical simulation at time $t$, i.e., we used data available in $[0,t]$. The left 
		column shows the flow fields including the free-surface for $z_2=0.5$ in cyan. The particle 
		centroid is shown 
		as a gray disc. The dashed line at $z_2=0.4$ marks the neighbourhood 
		$\cE:=\{z_1\in[-0.5,0.5], z_2\in[0.4,0.5]\}$ where we calculate the integrals for the 
		particle; red indicates the calculation is not active while green indicates an active 
		calculation of the NILE predictor as well as the full nonlinear predictor. The two rightmost 
		columns show the predictors (multiplied by a factor of 10 to increase the scale). The
		dashed blue line marks the zeros, i.e., we look for the first non-trivial zero crossing. The
		vertical dashed magenta line in the last time snapshot shows the true crossing of the particle
		from the active region past the threshold at $z_2=0.4$.}
\end{figure}

Having the full simulation data of the flow available, we want to check the predictive 
capabilities of the NILE balance condition $\cF_\sigma=0$ by computing $\cF_\sigma$ from 
the data. To make the problem harder, suppose we did not even know a reference trajectory 
inside the invariant manifold/free-surface
\benn
\cM=\{(z_1,z_2)\in\R^2:z_1\in[-0.5,0.5],z_2=0.5\}.
\eenn
Suppose we can only track the particle itself and have the local flow normal to $\cM$.
We project the particle trajectory $\tilde{\gamma}(t)$ along the direction normal to the wall
\benn
\gamma(t):=(0,\tilde{\gamma}_2(t))^\top\qquad \Rightarrow\quad \gamma(t)\subset \cM.
\eenn
This means we can calculate $\cF_\sigma(t)$ (using $t_0=0$) at fixed times according to
formula~\eqref{eq:Fsigma}. Clearly this predictor is locally linearized and only takes into
account the flow near $\cM$. To have a comparison, we consider a fully nonlinear balance
function that takes into account the local velocity of the particle at each point defined
by
\be
\cF_v(t):=\int_0^t v(\tilde{\gamma}(s))~\txtd s,
\ee
where $v(\tilde{\gamma}(s))$ is the actual particle velocity normal to the wall at each time point $s$.
To compute $\cF_v$ and $\cF_\sigma$, we only consider a region $\cE:=\{z_2\in[0.4,0.5]\}$ as shown in
Figure~\ref{fig:05}, i.e., outside of $\cE$, there is no contribution to the balance functions.
Figure~\ref{fig:05} shows the main results for the balance functions at four different time
snapshots. $\cF_\sigma$ based upon NILE performs extremely well considering the fact that only
local linearized flow near $\cM$ is taken into account. It slightly underestimates the true exit point
while using $\cF_v$ slightly overestimates it. In particular, there seems to be no significant performance
difference between the two balance functions, although one uses locally linear and the other
global fully nonlinear information.\medskip

It is important to point out that one may even try to use $\cF_\sigma$ as a predictor. Consider the 
third row in Figure~\ref{fig:05} at time $t=0.02914$. Since the balance function has just changed 
from increasing to decreasing, one may try to extrapolate to future times to predict the next time 
when $\cF_\sigma$ vanishes. This adds to the potential flexibility of the balance function
framework discussed here for many practical applications.

\section{Conclusion \& Outlook}
\label{sec:outlook}

In this paper we have proposed a framework to address the entry-exit relationship for trajectories
in dynamical systems with a strong focus on applications in fluid dynamics. Particles
travelling near invariant manifolds have a tremendous relevance for phenomena in fluids. We have
demonstrated that the concept of balance functions is well-suited to understand where particles are
expected to exit neighbourhoods of time-dependent invariant structures in fluid dynamics. Two model
problems have been used to determine that NILE balance functions provide a very efficient way to
merge theoretical precision as well as practical computation for particle trajectories. For the
Kuhlmann--Muldoon model, we showed how parameter dependencies can be uncovered using NILE balance
functions. Then we tested the concept on fully-resolved direct numerical simulation of a particle 
in a shear--stress-driven cavity. Again we obtained excellent agreement with the abstract theory. In
particular, local data near the invariant manifold and a projected particle trajectory were sufficient
to gain insight in the particle exit point and provided a practical strategy to predict it
based upon the current value and/or extrapolation of the balance function.\medskip

We stress that we only aimed here at demonstrating the main technical elements of balance functions,
and tried to show their applicability to key practical issues in fluid dynamics. Several open questions
remain after this study. On the one hand, further mathematical issues arise. For example, what happens
for balance functions in non-autonomous systems in the case of oscillatory instabilities, such as delayed
Hopf bifurcations~\cite{Neishtadt1,KuehnBook}. The question of random coefficients and change of stability
points~\cite{KuehnDelayUQ} as well as stochastic forcing~\cite{BerglundGentz,Kuske} are problems to be 
tackled in future work. Improving the link to the theory of Lagrangian coherent structures is another 
interesting direction to pursue so is the connection to purely set-valued approaches to dynamical systems. Furthermore, one might check whether balance functions can equivalently be defined using the Koopman operator, which is a strategy that already worked in other contexts for isochrons and isostable sets~\cite{MauroyMezicMoehlis}. 

With balance functions, many practical questions can now be tackled. The last 
example of the Navier-Stokes equation shows that working directly with data can be successful. To observe 
the results presented here in laboratory and field experiments would be of interest. Applications to
data analysis in oceanography and engineering are also conceivable as measurements near invariant
manifolds may potentially be easier to obtain in comparison to full monitoring of flow fields, e.g., 
in the context of flows in rivers as well as in flows near the coastline. 


\bibliographystyle{plain}
\bibliography{./KKR}

\begin{thebibliography}{10}

\bibitem{AulbachRasmussenSiegmund}
B.~Aulbach, M.~Rasmussen, and S.~Siegmund.
\newblock Invariant manifolds as pullback attractors of nonautonomous
  differential equations.
\newblock {\em Discr. Cont. Dyn. Syst.}, 15(2):579--597, 2006.

\bibitem{BaerErneuxRinzel}
S.M. Baer, T.~Erneux, and J.~Rinzel.
\newblock The slow passage through a {Hopf} bifurcation: delay, memory effects,
  and resonance.
\newblock {\em SIAM J. Appl. Math.}, 49(1):55--71, 1989.

\bibitem{Berger}
A.~Berger.
\newblock On finite-time hyperbolicity.
\newblock {\em Commun. Pure Appl. Anal.}, 10(3):963--981, 2011.

\bibitem{BergerDoanSiegmund1}
A.~Berger, T.S. Doan, and S.~Siegmund.
\newblock A definition of spectrum for differential equations on finite time.
\newblock {\em J. Differential Equat.}, 246:1098--1118, 2009.

\bibitem{BerglundGentz}
N.~Berglund and B.~Gentz.
\newblock {\em Noise-Induced Phenomena in Slow-Fast Dynamical Systems}.
\newblock Springer, 2006.

\bibitem{BranickiWiggins}
M.~Branicki and S.~Wiggins.
\newblock An adaptive method for computing invariant manifolds in
  non-autonomous, three-dimensional dynamical systems.
\newblock {\em Physica D}, 238(15):1625--1657, 2009.

\bibitem{ClarkeCarswell}
C.~Clarke and B.~Carswell.
\newblock {\em Principles of Astrophysical Fluid Dynamics}.
\newblock CUP, 2007.

\bibitem{CoxHsu}
R.G. Cox and S.K. Matthews.
\newblock The lateral migration of solid particles in a laminar flow near a
  plane.
\newblock {\em Int. J. Multiphase Flow}, 3(3):201--222, 1977.

\bibitem{DoanKarraschNguyenSiegmund}
T.S. Doan, D.~Karrasch, T.Y. Nguyen, and S.~Siegmund.
\newblock A unified approach to finite-time hyperbolicity which extends
  finite-time lyapunov exponents.
\newblock {\em J. Differential Equat.}, 252:5535--5554, 2012.

\bibitem{DucSiegmund}
L.H. Duc and S.~Siegmund.
\newblock Hyperbolicity and invariant manifolds for planar nonautonomous
  systems on finite time intervals.
\newblock {\em Int. J. Bif. Chaos}, 18(3):641--674, 2008.

\bibitem{DumortierRoussarie}
F.~Dumortier and R.~Roussarie.
\newblock {\em Canard Cycles and Center Manifolds}, volume 121 of {\em Memoirs
  Amer. Math. Soc.}
\newblock AMS, 1996.

\bibitem{Fenichel1}
N.~Fenichel.
\newblock Persistence and smoothness of invariant manifolds for flows.
\newblock {\em Indiana U. Math. J.}, 21:193--225, 1971.

\bibitem{Fenichel4}
N.~Fenichel.
\newblock Geometric singular perturbation theory for ordinary differential
  equations.
\newblock {\em J. Differential Equat.}, 31:53--98, 1979.

\bibitem{FroylandPadbergEnglandTreguier}
G.~Froyland, K.~Padberg, M.H. England, and A.M. Treguier.
\newblock Detection of coherent oceanic structures via transfer operators.
\newblock {\em Phys. Rev. Lett.}, 98(22):224503, 2007.

\bibitem{GogateBeenackersPandit}
P.R. Gogate, A.C.M Beenackers, and A.B. Pandit.
\newblock Multiple-impeller systems with a special emphasis on bioreactors: a
  critical review.
\newblock {\em Biochem. Eng. J.}, 6(2):109--144, 2000.

\bibitem{GreenRowleyHaller}
M.A. Green, C.W. Rowley, and G.~Haller.
\newblock Detection of {Lagrangian} coherent structures in three-dimensional
  turbulence.
\newblock {\em J. Fluid Mech.}, 572:111--120, 2007.

\bibitem{GH}
J.~Guckenheimer and P.~Holmes.
\newblock {\em Nonlinear Oscillations, Dynamical Systems, and Bifurcations of
  Vector Fields}.
\newblock Springer, New York, NY, 1983.

\bibitem{Haber}
S.~Haber.
\newblock A spherical particle moving slowly in a fluid with a radially varying
  viscosity.
\newblock {\em SIAM J. Appl. Math.}, 67(1):279--304, 2006.

\bibitem{Haller5}
G.~Haller.
\newblock Distinguished material surfaces and coherent structures in
  three-dimensional fluid flows.
\newblock {\em Physica D}, 149(4):248--277, 2001.

\bibitem{Haller4}
G.~Haller.
\newblock A variational theory of hyperbolic {Lagrangian} coherent structures.
\newblock {\em Physica D}, 240(7):574--598, 2011.

\bibitem{Haller3}
G.~Haller.
\newblock Lagrangian coherent structures.
\newblock {\em Ann. Rev. Fluid Mech.}, 47:137--162, 2015.

\bibitem{HallerSapsis1}
G.~Haller and T.~Sapsis.
\newblock Localized instability and attraction along invariant manifolds.
\newblock {\em SIAM J. Appl. Dyn. Syst.}, 9(2):611--633, 2010.

\bibitem{HallerYuan}
G.~Haller and G.~Yuan.
\newblock Lagrangian coherent structures and mixing in two-dimensional
  turbulence.
\newblock {\em Physica D}, 147(3):352--370, 2000.

\bibitem{HesthavenWarburton}
J.S. Hesthaven and T.~Warburton.
\newblock {\em Nodal Discontinuous Galerkin Methods: Algorithms, Analysis, and
  Applications}.
\newblock Springer, 2007.

\bibitem{HirschPughShub}
M.W. Hirsch, C.C. Pugh, and M.~Shub.
\newblock {\em Invariant Manifolds}.
\newblock Springer, 1977.

\bibitem{HofmannKuhlmann}
E.~Hofmann and H.C. Kuhlmann.
\newblock Particle accumulation on periodic orbits by repeated free surface
  collisions.
\newblock {\em Phys. Fluids}, 23:0721106, 2011.

\bibitem{Jones}
C.K.R.T. Jones.
\newblock Geometric singular perturbation theory.
\newblock In {\em Dynamical Systems (Montecatini Terme, 1994)}, volume 1609 of
  {\em Lect. Notes Math.}, pages 44--118. Springer, 1995.

\bibitem{JosephZenitHuntRosenwinkel}
G.G. Joseph, R.~Zenit, M.L. Hunt, and A.M. Rosenwinkel.
\newblock Particle-wall collisions in a viscous fluid.
\newblock {\em J. Fluid Mech.}, 433:329--346, 2001.

\bibitem{KarniadakisIsraeliOrszag}
G.E. Karniadakis, M.~Israeli, and S.A. Orszag.
\newblock High-order splitting methods for the incompressible {Navier-Stokes}
  equations.
\newblock {\em J. Comput. Phys.}, 97(2):414--443, 1991.

\bibitem{Karrasch1}
D.~Karrasch.
\newblock Linearization of hyperbolic finite-time processes.
\newblock {\em J. Differential Equat.}, 254:256--282, 2013.

\bibitem{KloedenRasmussen}
P.E. Kloeden and M.~Rasmussen.
\newblock {\em Nonautonomous Dynamical Systems}.
\newblock AMS, 2011.

\bibitem{KosinskiKosinskaHoffmann}
P.~Kosinski, A.~Kosinska, and A.C. Hoffmann.
\newblock Simulation of solid particles behaviour in a driven cavity flow.
\newblock {\em Powder Technol.}, 191(3):327--339, 2009.

\bibitem{KruSzm2}
M.~Krupa and P.~Szmolyan.
\newblock Relaxation oscillation and canard explosion.
\newblock {\em J. Differential Equat.}, 174:312--368, 2001.

\bibitem{KuehnUM}
C.~Kuehn.
\newblock Normal hyperbolicity and unbounded critical manifolds.
\newblock {\em Nonlinearity}, 27(6):1351--1366, 2014.

\bibitem{KuehnBook}
C.~Kuehn.
\newblock {\em Multiple Time Scale Dynamics}.
\newblock Springer, 2015.
\newblock 814 pp.

\bibitem{KuehnDelayUQ}
C.~Kuehn.
\newblock Uncertainty transformation via {Hopf} bifurcation in fast-slow
  systems.
\newblock {\em arXiv:1512.03002}, pages 1--19, 2015.

\bibitem{Kuhlmann}
H.C. Kuhlmann.
\newblock {\em Thermocapillary Convection in Models of Crystal Growth}, volume
  152 of {\em Springer Tracts in Modern Physics}.
\newblock Springer, Berlin, Heidelberg, 1999.

\bibitem{KuhlmannMuldoon2}
H.C. Kuhlmann and F.H. Muldoon.
\newblock Comment on ``{O}rdering of small particles in one-dimensional
  coherent structures by time-periodic flows''.
\newblock {\em Phys. Rev. Lett.}, 108:249401, 2012.

\bibitem{KuhlmannMuldoon}
H.C. Kuhlmann and F.H. Muldoon.
\newblock Particle-accumulation structures in periodic free-surface flows:
  Inertia versus surface collisions.
\newblock {\em Phys. Rev. E}, 85:046310, 2012.

\bibitem{KuhlmannMuldoon1}
H.C. Kuhlmann and F.H. Muldoon.
\newblock Comment on ``{S}ynchronization of finite-size particles by a
  traveling wave in a cylindrical flow'' [{P}hys. {F}luids 25, 092108 (2013)].
\newblock {\em Phys. Fluids}, 26(9):099101, 2014.

\bibitem{Kuske}
R.~Kuske.
\newblock Probability densities for noisy delay bifurcation.
\newblock {\em J. Stat. Phys.}, 96(3):797--816, 1999.

\bibitem{LekienRoss}
F.~Lekien and S.D. Ross.
\newblock The computation of finite-time {Lyapunov} exponents on unstructured
  meshes and for {non-Euclidean} manifolds.
\newblock {\em Chaos}, 20(1):017505, 2010.

\bibitem{LiuXuGao}
X.~Liu, G.~Xu, and S.~Gao.
\newblock Micro fluidized beds: Wall effect and operability.
\newblock {\em Chem. Eng. J.}, 137(2):302--307, 2008.

\bibitem{LuoMaxeyKarniadakis}
X.~Luo, M.R. Maxey, and G.E. Karniadakis.
\newblock Smoothed profile method for particulate flows: error analysis and
  simulations.
\newblock {\em J. Comput. Phys.}, 228(5):1750--1769, 2009.

\bibitem{MauroyMezicMoehlis}
A.~Mauroy, I.~Mezic, and J.~Moehlis.
\newblock Isostables, isochrons, and {Koopman} spectrum for the action-angle
  reduction of stable fixed point dynamics.
\newblock {\em Physica D}, 261:19--30, 2013.

\bibitem{MaxeyRiley}
M.R. Maxey and J.J. Riley.
\newblock Equation of motion for a small rigid sphere in a nonuniform flow.
\newblock {\em Phys. Fluids}, 26(4):883--889, 1983.

\bibitem{McCave}
I.N. McCave.
\newblock Size spectra and aggregation of suspended particles in the deep
  ocean.
\newblock {\em Deep Sea Res. A}, 31(4):329--352, 1984.

\bibitem{MelnikovPushkinShevtsova}
D.E. Melnikov, D.O. Pushkin, and V.M. Shevtsova.
\newblock Synchronization of finite-size particles by a traveling wave in a
  cylindrical flow.
\newblock {\em Phys. Fluids}, 25(9):092108, 2013.

\bibitem{MuldoonKuhlmann}
F.H. Muldoon and H.C. Kuhlmann.
\newblock Coherent particulate structures by boundary interaction of small
  particles in confined periodic flows.
\newblock {\em Physica D}, 253:40--65, 2013.

\bibitem{NakayamaYamamoto}
Y.~Nakayama and R.~Yamamoto.
\newblock Simulation method to resolve hydrodynamic interactions in colloidal
  dispersions.
\newblock {\em Phys. Rev. E}, 71:036707, 2005.

\bibitem{Neishtadt1}
A.I. Neishtadt.
\newblock Persistence of stability loss for dynamical bifurcations. {I}.
\newblock {\em Differential Equations Translations}, 23:1385--1391, 1987.

\bibitem{Neishtadt2}
A.I. Neishtadt.
\newblock Persistence of stability loss for dynamical bifurcations. {II}.
\newblock {\em Differential Equations Translations}, 24:171--176, 1988.

\bibitem{OttinoKhakhar}
J.M. Ottino and D.V. Khakhar.
\newblock Mixing and segregation of granular materials.
\newblock {\em Ann. Rev. Fluid Mech.}, 32:55--91, 2000.

\bibitem{PeacockDabiri}
T.~Peacock and J.~Dabiri.
\newblock Introduction to focus issue: {Lagrangian} coherent structures.
\newblock {\em Chaos}, 10(1):017501, 2010.

\bibitem{PowerPower}
H.~Power and B.~Febres de~Power.
\newblock Second-kind integral equation formulation for the slow motion of a
  particle of arbitrary shape near a plane wall in a viscous fluid.
\newblock {\em SIAM J. Appl. Math.}, 53(1):60--70, 1993.

\bibitem{PushkinMelnikovShevtsova}
D.O. Pushkin, D.E. Melnikov, and V.M. Shevtsova.
\newblock Ordering of small particles in one-dimensional coherent structures by
  time-periodic flows.
\newblock {\em Phys. Rev. Lett.}, 106:234501, 2011.

\bibitem{Rainer}
A.~Rainer.
\newblock Differentiable roots, eigenvalues, and eigenvectors.
\newblock {\em Israel J. Math.}, 201(1):99--122, 2014.

\bibitem{Rasmussen}
M.~Rasmussen.
\newblock {\em Attractivity and Bifurcation for Nonautonomous Dynamical
  Systems}.
\newblock Springer, 2007.

\bibitem{Rasmussen1}
M.~Rasmussen.
\newblock Finite-time attractivity and bifurcation for nonautonomous
  differential equations.
\newblock {\em Differential Equations Dynam. Systems}, 18(1):57--78, 2010.

\bibitem{RomanoKuhlmann}
F.~Roman{\`o} and H.C. Kuhlmann.
\newblock Interaction of a finite‐size particle with the moving lid of a
  cavity.
\newblock {\em Proc. Appl. Math. Mech.}, 15(1):519--520, 2015.

\bibitem{RomanoKuhlmann1}
F.~Roman{\`o} and H.C. Kuhlmann.
\newblock Smoothed-profile method for momentum and heat transfer in particulate
  flows.
\newblock {\em Int. J. Num. Meth. Fluids}, pages 1--28, 2016.
\newblock early view, DOI: 10.1002/fld.4279.

\bibitem{RosebrockOkeCarroll}
U.~Rosebrock, P.R. Oke, and G.~Carroll.
\newblock An application framework for the rapid deployment of ocean models in
  support of emergency services: application to the {MH370} search.
\newblock In {\em Environmental Software Systems. Infrastructures, Services and
  Applications}, pages 235--241. Springer, 2015.

\bibitem{RubinJonesMaxey}
J.~Rubin, C.K.R.T. Jones, and M.~Maxey.
\newblock Settling and asymptotic motion of aerosol particles in a cellular
  flow field.
\newblock {\em J. Nonlinear Sci.}, 5:337--358, 1995.

\bibitem{RubinowKeller}
S.I. Rubinow and J.B. Keller.
\newblock The transverse force on a spinning sphere moving in a viscous fluid.
\newblock {\em J. Fluid Mech.}, 11(3):447--459, 1961.

\bibitem{Schecter}
S.~Schecter.
\newblock Persistent unstable equilibria and closed orbits of a singularly
  perturbed equation.
\newblock {\em J. Differential Equat.}, 60:131--141, 1985.

\bibitem{SchwabeMizevUdhayasankarTanaka}
D.~Schwabe, A.I. Mizev, M.~Udhayasankar, and S.~Tanaka.
\newblock Formation of dynamic particle accumulation structures in oscillatory
  thermocapillary flow in liquid bridges.
\newblock {\em Phys. Fluids}, 19:072102, 2007.

\bibitem{ScrivenSternling}
L.E. Scriven and C.V. Sternling.
\newblock The {Marangoni} effects.
\newblock {\em Nature.}, 187:186--188, 1960.

\bibitem{ShaddenLekienMarsden}
S.C. Shadden, F.~Lekien, and J.E. Marsden.
\newblock Definition and properties of {Lagrangian} coherent structures from
  finite-time {Lyapunov} exponents in two-dimensional aperiodic flows.
\newblock {\em Phys. D}, 212(3):271--304, 2005.

\bibitem{StoneStroockAjdari}
H.A. Stone, A.D. Stroock, and A.~Ajdari.
\newblock Engineering flows in small devices: microfluidics toward a
  lab-on-a-chip.
\newblock {\em Annu. Rev. Fluid Mech.}, 36:381--411, 2004.

\bibitem{TanakaKawamuraUenoSchwabe}
S.~Tanaka, H.~Kawamura, I.~Ueno, and D.~Schwabe.
\newblock Flow structure and dynamic particle accumulation in thermocapillary
  convection in a liquid bridge.
\newblock {\em Phys. Fluids}, 18:067103, 2006.

\bibitem{VasseurCox}
R.~Vasseur and R.G. Cox.
\newblock The lateral migration of spherical particles sedimenting in a
  stagnant bounded fluid.
\newblock {\em J. Fluid Mech.}, 80(3):561--591, 1977.

\bibitem{Wechselberger1}
M.~Wechselberger.
\newblock A propos de canards (apropos canards).
\newblock {\em Trans. Amer. Math. Soc.}, 364:3289--3309, 2012.

\bibitem{XuYu}
B.H. Xu and A.B. Yu.
\newblock Numerical simulation of the gas-solid flow in a fluidized bed by
  combining discrete particle method with computational fluid dynamics.
\newblock {\em Chem. Eng. Sci.}, 52(16):2785--2809, 1997.

\bibitem{YoungLeeming}
J.~Young and A.~Leeming.
\newblock A theory of particle deposition in turbulent pipe flow.
\newblock {\em J. Fluid Mech.}, 340:129--159, 1997.

\bibitem{ZhangKleinstreuerKimCheng}
Z.~Zhang, C.~Kleinstreuer, C.S. Kim, and Y.S. Cheng.
\newblock Vaporizing microdroplet inhalation, transport, and deposition in a
  human upper airway model.
\newblock {\em Aerosol Sci. Tech.}, 38(1):36--49, 2004.

\end{thebibliography}

\end{document}